\documentclass[12pt,oneside]{amsart}
\usepackage[english]{babel} 
\usepackage{amsmath} 
\usepackage{amsthm}
\usepackage{amssymb}
\usepackage{geometry}
\usepackage{nicefrac} 
\usepackage{hyperref}
\usepackage{graphicx}
\usepackage{braket}
\usepackage{subfig}
\usepackage{csquotes}
\usepackage{color}
\providecommand{\abs}[1]{\left|#1\right|}
\providecommand{\norm}[1]{\left \| #1\right \|}

\newcommand{\I}{\mathcal{I}}

\newcommand{\R}{\mathbb{R}}
\newcommand{\N}{\mathbb{N}}

\numberwithin{equation}{section}

\theoremstyle{plain}
\newtheorem{theorem}{Theorem}[section]

\theoremstyle{plain}
\newtheorem{definition}[theorem]{Definition}

\theoremstyle{plain}
\newtheorem{lemma}[theorem]{Lemma}

\theoremstyle{plain}
\newtheorem{corollary}[theorem]{Corollary}

\theoremstyle{plain}
\newtheorem{proposition}[theorem]{Proposition}

\theoremstyle{remark}
\newtheorem{remark}[theorem]{Remark}

\title[]{Maximum principles and related problems for a class of nonlocal extremal operators}
\date{}
\author{Isabeau Birindelli, Giulio Galise, Delia Schiera}
\address{Dipartimento di Matematica Guido Castelnuovo, Sapienza
Universit\`a di Roma, Piazzale Aldo Moro 5, 00185, Roma, Italy.}
\email[I. Birindelli]{isabeau@mat.uniroma1.it}
\email[G. Galise]{galise@mat.uniroma1.it}
\email[D. Schiera]{delia.schiera@uniroma1.it}
\thanks{The authors were partially supported by INdAM-GNAMPA}

\begin{document}
\maketitle

 \begin{abstract}We study the validity of the comparison and maximum principles, and their relation with principal eigenvalues, for a class of degenerate nonlinear  operators that are extremal among operators with one dimensional fractional diffusion.\end{abstract}

\section{Introduction}
The fractional Laplacian is a singular integral operator defined e.g. by
$$(-\Delta)^su(x):=-\frac12 C_{N,s} \int_{\R^N} \frac{\delta(u, x, y)}{|y|^{N+2s}} \, dy$$
with $s \in (0,1)$ and 
$$
\delta(u, x, y)= u(x+y)+ u(x-y)-2u(x), $$
so that the value of $(-\Delta)^su$ at $x$ depends on the value of $u$ in the whole of $\R^N$. But, of course, it is possible to define singular integral operators that depend only on subdimensional sets of $\R^N$. For example, one can consider 1-dimensional sets, fixing a direction $\xi  \in \R^N$ and letting
$$ \mathcal{I}_\xi u(x) :=C_s \int_{0}^{+\infty} \frac{\delta(u, x, \tau \xi)}{\tau^{1+2s}} \, d\tau. $$
Here $C_s=C_{1,s}$ so that  $ \mathcal{I}_\xi u(x)$ acts as the $2s$-fractional derivative of $u$ in  the direction $\xi$.
Hence, we can  denote $\mathcal{V}_k$ the family of $k$-dimensional orthonormal sets in $\R^N$ and define
the following nonlocal nonlinear operators 
\[ \I_k^+ u(x) := \sup \left \{ \sum_{i=1}^k \mathcal{I}_{\xi_i} u(x) \colon \{ \xi_i \}_{i=1}^k \in \mathcal{V}_k \right \} \]
\[ \I_k^- u(x) := \inf \left \{ \sum_{i=1}^k \mathcal{I}_{\xi_i} u(x) \colon \{ \xi_i \}_{i=1}^k \in \mathcal{V}_k \right \}. \]
These operators have been very recently considered in \cite{BGT}, where representation formulas were given, and  
in \cite{DelPezzoQuaasRossi}, where the operators $\I_1^\pm$ are shown to be related with a notion of fractional convexity. 
These extremal operators, even for $k=N$, are intrinsically different from the fractional Laplacian and we will show some new phenomena arising. 
We concentrate in particular on exterior Dirichlet problems in bounded domains.

Precisely, for $\Omega$ a bounded domain of $\R^N$, we will study:
\begin{equation}\label{07eq}
\left\{\begin{array}{cl}
\I^\pm_ku(x)+c(x)u(x)=f(x) &  \text{in $\Omega$}\\
u=0 & \text{in $\R^N\backslash\Omega$}.
\end{array}\right.
\end{equation}

The first difference we wish to emphasize is that in general these operators are not continuous, precisely, even if $u$ is in $C^\infty(\Omega)$ and bounded, $\I_k^\pm u(\cdot)$ may not be continuous. What is required in order to have continuity, or lower or upper semicontinuity,  is a global condition on the regularity of $u$;  this will be shown in Proposition \ref{semicont}. 
This is a striking difference with respect to the case of nonlinear integro-differential operators like e.g.~the ones considered in \cite{CaffarelliSilvestre}, which are continuous once $C^{1, 1}$ regularity holds in the domain $\Omega$. 
These continuity properties play a key role in the arguments used for the proofs of the comparison principle, Alexandrov-Bakelman-Pucci estimate, and the Harnack inequality, showing that the setting we are interested in deviates in a substantial way from \cite{CaffarelliSilvestre}. 

Nevertheless, we will show that the comparison principle still holds for $\I_k^\pm$ in any bounded domain;
we recall that a comparison principle for $\I_1^\pm$ was also proved in \cite{DelPezzoQuaasRossi}, but under the assumption that the domain 
is strictly convex. We wish to remark that in fact the comparison principle here is very simple compared to the local case. As it is well known, in the theory of viscosity solutions
the comparison principle for second order operators requires the Jensen-Ishii's lemma, see \cite{CIL}, which in turn lies on a remarkably complex proof that uses tools from convex analysis. 
Here, instead, the proof is completely self contained and uses only a straightforward calculation, somehow more similar to the case of first order local equations, where just the doubling variable technique is used.

Via an adaptation of the Perron's method by \cite{CIL}, the comparison principle allows to prove existence of solutions
for \eqref{07eq}. Let us mention that existence in a very general setting that includes elliptic integro-differential operators was proved in \cite{BCI, BarlesImbert}.
However the approach we use is quite immediate, and it seemed to us simpler and friendlier to the reader to just give the proof then checking if we fit into the general Barles-Chasseigne-Imbert setting.

We conclude with  the proof of H\"older estimates for $\I_1^\pm$ in uniformly convex domains and the validity of maximum principle for the operators
\[ \I_k^\pm\cdot +\mu \cdot \]
with $\mu$ below the generalized principal eigenvalues, which, adapting the classical definition in \cite{BNV}, we set as
$$ \mu_k^\pm = \sup \{ \mu \colon \exists v \in LSC(\Omega)\cap L^\infty(\R^N), v>0 \text{ in } \Omega, v \ge 0 \text{ in } \R^N, \I_k^\pm v + \mu v \le 0 \text{ in } \Omega \}. $$

Let us mention that with our choice of the constant $C_s$, the operators $\I_k^\pm$ converge to the operators $\mathcal{P}_k^\pm$, the so called truncated Laplacians, defined by
\[  \mathcal{P}^+_k (D^2u)(x):= \sum_{i=N-k+1}^{N} \lambda_i(D^2 u(x)) = \max \left\{ \sum_{i=1}^k \langle D^2 u(x) \xi_i, \xi_i \rangle \, \colon \, \{ \xi_i\}_{i=1}^k \in \mathcal{V}_k \right\} \]
and
\[  \mathcal{P}^-_k (D^2u)(x):= \sum_{i=1}^{k} \lambda_i(D^2 u(x)) = \min \left\{ \sum_{i=1}^k \langle D^2 u(x) \xi_i, \xi_i \rangle \, \colon \, \{ \xi_i\}_{i=1}^k \in \mathcal{V}_k \right\}, \]
where  $\lambda_i(D^2 u)$ are the eigenvalues of $D^2u$ arranged in nondecreasing order, see \cite{CaffarelliLiNirenberg, HarveyLawson, BGI1, BGI}. 
Of course there are other classes of nonlocal operators that approximate  $\mathcal{P}^\pm_k (D^2 u)(x)$, as can be seen in \cite{BGT}. But we have concentrated on those that are somehow more of a novelty.

In general we wish to emphasize that in this setting we have differences both with the local equivalent operators and with more standard nonlocal operators. We have already seen that they are in general not continuous, also it is immediate that even when $k=N$, which in the local case gives  $\mathcal{P}^+_N (D^2u)(x)=\mathcal{P}^-_N (D^2u)(x)=\Delta u$, it is not true that $\I_N^-$ is equal to $\I_N^+$ or that it is equal to the fractional Laplacian. But there are other differences, for example regarding the validity of the strong maximum principle, see Theorem \ref{SMP}, or the fact that its validity depends also on the positivity of the solution outside the domain, see Proposition \ref{prop:strong}, or regarding the fact that for $\mathcal{P}^\pm_k$ the  supremum (infimum) among all possible $k$-dimensional frames is in fact a maximum (minimum), while here the extremum may not be reached as it is shown in the examples before Proposition \ref{semicont}. Hence we encourage the reader to pursue her reading in order to see all these fascinating differences.

\vspace{0.5cm}

\noindent This paper is organized as follows. 

After a preliminary section,  in Section \ref{sec:continuity} we study continuity properties of $\I_k^\pm$. We will first give counterexamples showing that in general these operators are not continuous, and then we prove that they preserve upper (or lower) semicontinuity under some global assumptions. As a related result, we also show that the supremum and the infimum in the definitions of $\I_k^\pm$ are in general not attained. 

Section \ref{sec:comparison} is devoted to the proof of the comparison principle. We investigate the validity and the failure of strong maximum/minimum principles for these operators. Moreover, we prove a Hopf-type lemma for $\I_N^-$ and $\I_k^+$. 

In Section \ref{sec:Perron} we exploit the uniform convexity of the domain $\Omega$ to construct first barrier functions, then solutions for the Dirichlet problem by using the Perron's method \cite{CIL}.

Section \ref{sec:maximum} is devoted to the analysis of validity of the  maximum principle for $\I_k^\pm\cdot +\mu \cdot$, and to the relation with principal eigenvalues. 

Finally, H\"older estimates for solutions of 
$\I_1^\pm u= f$ in $\Omega$, $u=0$ in $\R^N \setminus \Omega$, where $\Omega$ is a uniformly convex domain, are proved in Section \ref{sec:holder}. 

We will use them in Section \ref{sec:existence} to prove existence of a positive principal eigenfunction.  

\section*{Notations}
\renewcommand{\arraystretch}{1.3}
\begin{tabular}{cp{0.8\textwidth}}
  $B_r(x)$ & ball centered in $x$ of radius $r$ \\
  $\mathcal{S}^{N-1}$ & unitary sphere in $\R^{N}$\\
  $\{ e_i\}_{i=1}^N$ & canonical basis of $\R^N$\\
  $d(x)$& $= \inf_{y \in \partial \Omega} \abs{x-y}$, the distance function from $x \in \Omega$ to $\partial \Omega$ \\
  $LSC(\Omega)$ & space of lower semi continuous functions on $\Omega$ \\
  $USC(\Omega)$ & space of upper semi continuous functions on $\Omega$ \\
  $\delta(u, x, y)$&$= u(x+y)+ u(x-y)-2u(x)$  \\
    $\mathcal{I}_\xi u(x)$& $=C_s \int_{0}^{+\infty} \frac{\delta(u, x, \tau \xi)}{\tau^{1+2s}} \, d\tau$, where $\xi \in \mathcal{S}^{N-1}$ and $C_s$ is a normalizing constant \\
  $\hat x$ & $=\frac{x}{\abs{x}}$ \\
   $\beta(a, b)$ & $=\int_0^1 t^{-b} (1-t)^{-a} \, dt$ \\
   $\mathcal{V}_k$ & the family of $k$-dimensional orthonormal sets in $\R^N$
\end{tabular}

\section{Preliminaries}\label{sec:prelim}
We recall the definition of viscosity solution in this nonlocal context \cite{BCI, BarlesImbert}. For definitions and main properties of viscosity solutions in the classical local framework we refer to the survey \cite{CIL}. 
\begin{definition}  
Given a function $f \in C(\Omega \times \R)$, we say that $u \in L^\infty(\R^N) \cap LSC(\Omega) $ (respectively $USC(\Omega)$) is a (viscosity) supersolution (respectively subsolution) to 
\begin{equation}\label{eq def sol} \I_k^+ u +f(x, u(x)) = 0 \text{ in } \Omega \end{equation}
if for every point $x_0 \in \Omega$ and every function $\varphi \in C^2(B_\rho(x_0))$, $\rho >0$, such that $x_0$ is a minimum (resp. maximum) point to $u - \varphi$, then 
\begin{equation}\label{cond super} \mathcal{I}(u, \varphi, x_0, \rho) +f(x_0, u(x_0)) \le 0 \quad \text{(resp. $\ge 0$)}\end{equation}
where 
\[ \mathcal{I}(u, \varphi, x_0, \rho) = C_s \sup_{\{\xi_i \} \in \mathcal{V}_k} \sum_{i=1}^k  \left (\int_{0}^\rho \frac{\delta(\varphi, x_0, \tau \xi_i)}{\tau^{1+2s}} \, d\tau  + \int_{\rho}^{+\infty} \frac{\delta(u, x_0, \tau \xi_i)}{\tau^{1+2s}} \, d\tau \right). \]
 We say that a continuous function $u$ is a  solution of \eqref{eq def sol} if it is both a supersolution and a subsolution of \eqref{eq def sol}.
We analogously define viscosity sub/super solutions for the operator $\I_k^-$, taking the infimum  over $\mathcal{V}_k$  in place of the supremum. 
\end{definition}

\begin{remark}
We stress that the definition above is derived from $-(-\Delta)^s$, that means, a minus sign in front of the operator is taken into account. 
\end{remark}

\begin{remark}\label{test}
In the definition of supersolution above we can assume without loss of generality that $u > \varphi$ in $B_\rho(x_0) \setminus \{ x_0 \}$, and $\varphi (x_0) = u(x_0)$. Indeed, let us assume that for any such $\varphi$ 
\[ C_s \sup_{\{\xi_i \} \in \mathcal{V}_k} \sum_{i=1}^k  \left (\int_{0}^\rho \frac{\delta(\varphi, x_0, \tau \xi_i)}{\tau^{1+2s}} \, d\tau  + \int_{\rho}^{+\infty} \frac{\delta(u, x_0, \tau \xi_i)}{\tau^{1+2s}} \, d\tau \right) + f(x_0, u(x_0)) \le 0 \]
is satisfied, and consider a general $\tilde \varphi \in C^2(B_\rho(x_0))$ such that $u-\tilde \varphi$ has a minimum in $x_0$. 
We take for any $n \in \N$ 
\[ \varphi_n(x)=\tilde \varphi(x) + u(x_0) - \tilde \varphi(x_0) - \frac 1n \abs{x-x_0}^2, \]
and notice that $u(x_0)=\varphi_n(x_0)$, and since $u(x_0) - \tilde \varphi(x_0) \le u(x)-\tilde \varphi(x)$, 
\[ \varphi_n(x) \le u(x) - \frac 1n \abs{x-x_0}^2 < u(x) \]
for any $x \in B_\rho(x_0) \setminus \{ x_0 \}$. Also, for any $n \in \N$, 
\begin{multline*}
C_s \sup_{\{\xi_i \} \in \mathcal{V}_k} \sum_{i=1}^k  \left (\int_{0}^\rho \frac{\delta(\tilde \varphi, x_0, \tau \xi_i)}{\tau^{1+2s}} \, d\tau  + \int_{\rho}^{+\infty} \frac{\delta(u, x_0, \tau \xi_i)}{\tau^{1+2s}} \, d\tau\right) \\+ f(x_0, u(x_0)) \le C_s \frac{k \rho^{2-2s}}{n (1-s)},\end{multline*}
and the conclusion follows taking  the limit $n \to \infty$. 
\end{remark}
\begin{remark}
We point out that if we verify \eqref{cond super} for $\rho_1$, then it is also verified for any $\rho_2 > \rho_1$, since
\[ \mathcal{I}(u, \varphi, x_0, \rho_2)  \le \mathcal{I}(u, \varphi, x_0, \rho_1). \]
\end{remark}
\begin{remark}
The operators $\I_k^\pm$ satisfy the following ellipticity condition: if $\psi_1, \psi_2 \in C^2(B_\rho(x_0)) \cap L^{\infty}(\R^N)$ for some $\rho>0$ are such that $\psi_1-\psi_2$ has a maximum in $x_0$, then 
\[ \I_k^\pm \psi_1 (x_0) \le \I_k^\pm \psi_2 (x_0). \]
Indeed, if $\psi_1(x_0)-\psi_2(x_0) \ge \psi_1(x)-\psi_2(x)$, for all $x \in B_\rho(x_0)$ then 
\[ \delta(\psi_1, x_0, \tau \xi_i) \le \delta(\psi_2, x_0, \tau \xi_i) \]
which yields the conclusion. 
\end{remark}
\begin{remark}
Notice that in the definition above we assumed $u \in L^\infty(\R^N)$, as this will be enough for our purposes, however, one can also consider unbounded functions $u$ with a suitable growth condition at infinity, see \cite{BGT}. 
\end{remark}

\section{Continuity}\label{sec:continuity}
In this section  we study continuity properties of the maps $x \mapsto \I_k^\pm u(x)$. We start by showing that the assumption $u \in C^2(\Omega) \cap L^\infty(\R^N)$ which ensures that $\I_k^\pm u(x)$ is well defined, is in fact not enough to guarantee the continuity of $\I_k^\pm u(x)$ with respect to $x$. 
What is needed is a more global assumption as it will be shown later. 

Let $u$ be the function defined as follows:
\begin{equation}\label{eq5}
u(x)=\left\{\begin{array}{rl}
0 &  \text{if $|x|\leq1$ or $\langle x, e_N \rangle \leq0$}\\
-1 & \text{otherwise.}
\end{array}\right.
\end{equation}
Set $\Omega=B_1(0)$. The map
$$
x\in\Omega\mapsto\I^+_ku(x)
$$
is well defined, since $u$ is bounded in $\R^N$ and smooth (in fact constant) in $\Omega$. We shall prove that it is not continuous at $x=0$ when $k<N$.

Let us first compute the value $\I_k^+u(0)$. Since $u\leq0$ in $\R^N$ it turns out that for any $|\xi|=1$ 
\begin{equation*}
\I_\xi u(0)=C_s  \int_0^{+\infty}\frac{u(\tau\xi)+u(-\tau\xi)}{\tau^{1+2s}}\,d\tau\leq0.
\end{equation*}
Hence
\begin{equation}\label{eq1}
\sup_{\left\{\xi_i\right\}_{i=1}^k\in{\mathcal V}_k}\sum_{i=1}^k\I_{\xi_i} u(0)\leq0.
\end{equation}
On the other hand, choosing the first $k$-unit vectors $e_1,\ldots,e_k$ of the standard basis, we obtain that 
\begin{equation}\label{eq2}
\I_{e_1}u(0)=\ldots=\I_{e_k}u(0)=0.
 \end{equation}
Hence by \eqref{eq1}-\eqref{eq2} 
\[
\I^+_ku(0)=0.
 \]
 
 \begin{figure}
\centering
\includegraphics[width=0.7\textwidth]{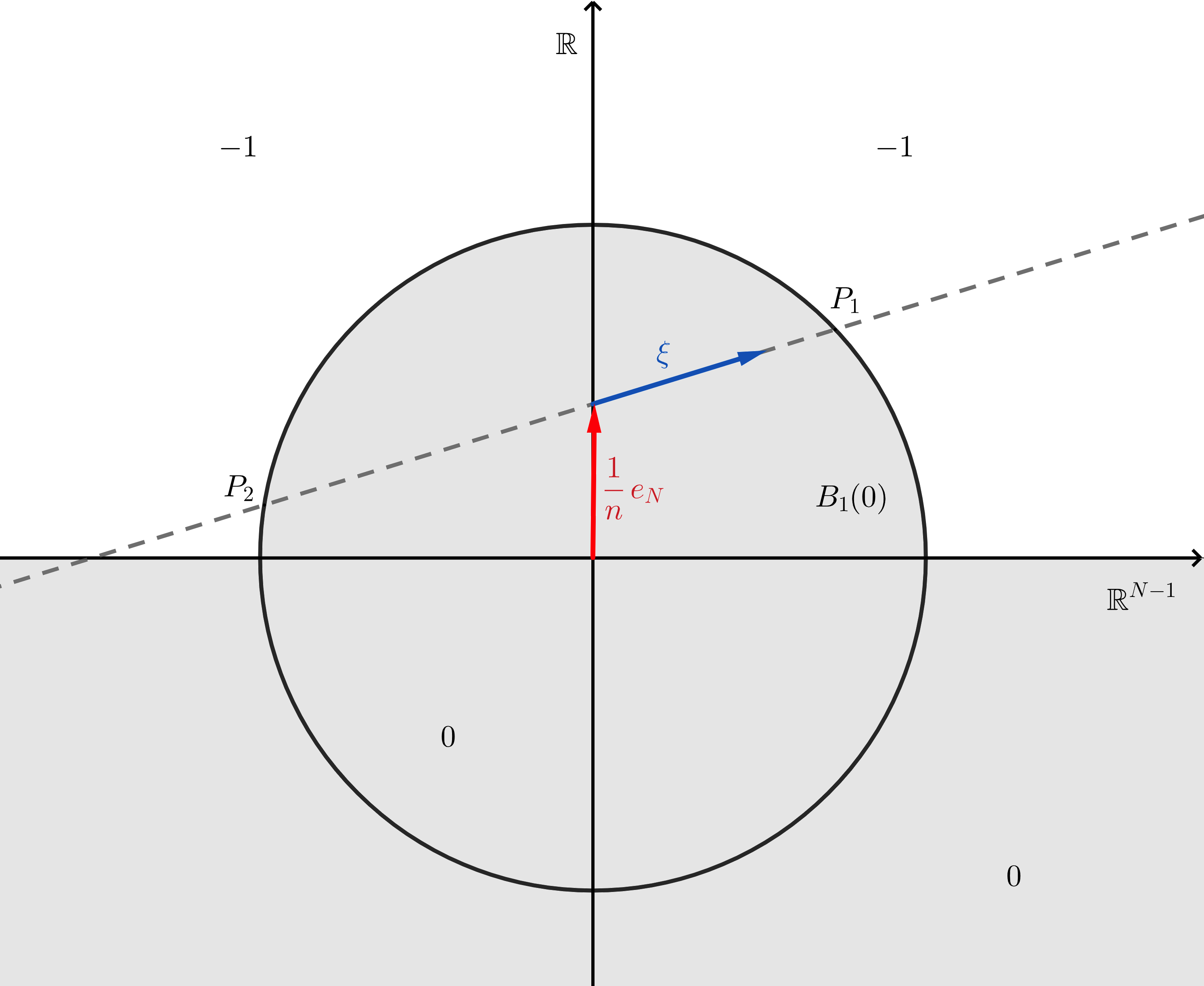}
\caption{We represent with $P_1$ the point $\frac1n e_N + \tau_1(n) \xi$, whereas $P_2=\frac1n e_N - \tau_2(n) \xi$.}
\label{cont}
\end{figure}

Now we are going to prove that 
$$
\limsup_{n\to+\infty}\I^+_ku\left(\frac1ne_N\right)<0
$$
where $e_N=(0,\ldots,0,1)$. 
Fix any $|\xi|=1$. Since $\I_\xi u=\I_{-\xi}u$, we can further assume that $\left\langle \xi,e_N\right\rangle\geq0$. Then, for any $n>1$,
\begin{equation}\label{eq4}
\begin{split}
\I_\xi u\left(\frac1n e_N\right)&=C_s  \int_0^{+\infty}\frac{u(\frac1ne_N+\tau\xi)+u(\frac1ne_N-\tau\xi)}{\tau^{1+2s}}\,d\tau\\
&=C_s \left(-\int_{\tau_1(n)}^{\tau_2(n)}\frac{1}{\tau^{1+2s}}\,d\tau+\int_{\tau_2(n)}^{+\infty}\frac{-1+u(\frac1ne_N-\tau\xi)}{\tau^{1+2s}}\,d\tau \right)
\end{split}
\end{equation}
where 
\[ \tau_1(n)=-\frac{\left\langle \xi,e_N\right\rangle}{n}+\sqrt{\left(\frac{\left\langle \xi,e_N\right\rangle}{n}\right)^2+1-\frac{1}{n^2}} \] and 
\[ \tau_2(n)=\frac{\left\langle \xi,e_N\right\rangle}{n}+\sqrt{\left(\frac{\left\langle \xi,e_N\right\rangle}{n}\right)^2+1-\frac{1}{n^2}}. \]
Notice that if $\tau \le \tau_1(n)$ then $\frac1n e_N \pm \tau \xi \in \overline{B_1(0)}$, if $\tau \in (\tau_1(n), \tau_2(n)]$ then $\frac1n e_N - \tau \xi \in \overline{B_1(0)}$, however $\frac1n e_N + \tau \xi \not \in \overline{B_1(0)}$. Finally, if $\tau > \tau_2(n)$, then $\frac1n e_N \pm \tau \xi \not \in \overline{B_1(0)}$, see also Figure \ref{cont}. 

Using $u\leq0$ we obtain from \eqref{eq4} that 
$$
\I_\xi u\left(\frac1n e_N\right)\leq - C_s \int_{\tau_1(n)}^{+\infty}\frac{1}{\tau^{1+2s}}\,d\tau.
$$
Moreover, since $\tau_1(n)\leq\sqrt{1-\frac{1}{n^2}}$\ , we infer that 
\begin{equation}\label{eq6}
\I_\xi u\left(\frac1n e_N\right)\leq -C_s  \int_{\sqrt{1-\frac{1}{n^2}}}^{+\infty}\frac{1}{\tau^{1+2s}}\,d\tau=-C_s  \frac{1}{2s{(1-\frac{1}{n^2})}^s}
\end{equation}
for any $|\xi|=1$. Then
$$
\I_k^+u\left(\frac1n e_N\right)\leq -\frac{kC_s }{2s{(1-\frac{1}{n^2})}^s}
$$
and
$$
\limsup_{n\to+\infty}\I_k^+u\left(\frac1n e_N\right)\leq -\frac{kC_s }{2s}<0
$$
as we wanted to show.

A slight modification of the function $u$ in \eqref{eq5} allows us to show that the map
$$
x\in\Omega\mapsto\I^+_Nu(x) 
$$ 
is also, in general, not continuous. \\
Consider the function
\begin{equation*}
u(x)=\left\{\begin{array}{rl}
0 &  \text{if $|x|\leq1$, or $\langle x, e_N \rangle \leq0$ or $\sum_{i=1}^{N-1} \langle x, e_i \rangle ^2=0$}\\
-1 & \text{otherwise.}
\end{array}\right.
\end{equation*}
As before, using the fact that $u\leq0$ in $\R^N$ and that 
$$
\I_{e_1}u(0)=\ldots=\I_{e_N}u(0)=0,
$$
we have 
$$
\I_N^+u(0)=0.$$
Moreover for any $|\xi|=1$ such that $\left\langle \xi,e_N\right\rangle\in[0,1)$, then \eqref{eq6} still holds. Since for any orthonormal basis $\left\{\xi_1,\ldots,\xi_N\right\}$ there is at most one $\xi_i$ such that $\left\langle \xi_i,e_N\right\rangle=1$, then 
$$
\I_N^+u\left(\frac1n e_N\right)\leq -C_s  \frac{N-1}{2s{(1-\frac{1}{n^2})}^s}
$$
and
$$
\limsup_{n\to+\infty}\I_N^+u\left(\frac1n e_N\right)\leq -C_s \frac{N-1}{2s}.
$$ 

\vskip3ex

A further consequence of the lack of continuity is that the $\sup$ or $\inf$ in the definition of $\I_k^\pm$ are in general not attained under the only assumption $u \in C^2(\Omega) \cap L^\infty(\R^N)$. 
As an example, take
\[ 
u(x)=\begin{cases}
0 &\text{ if} \abs{x} \le 1 \text{ or } \langle x, e_N \rangle  \le 0 \\
e^{-\langle x, e_N \rangle } &\text{ otherwise.}
\end{cases} \]
Then 
\[ \I_1^+ u(0)= \sup_{\abs{\xi}=1} \I_\xi(0)=C_s \sup_{\abs{\xi}=1} \int_0^{+\infty} \frac{u(\tau \xi) + u(-\tau \xi)}{\tau^{1+2s}} \, d\tau. \]
Since $\I_\xi u(0)=\I_{-\xi} u(0)$, we can assume without loss of generality that $\langle \xi, e_N \rangle \in [0, 1]$. Thus
\[ \I_1^+ u(0) = C_s \sup_{\abs{\xi}=1, \langle \xi, e_N \rangle \ge 0} \int_0^{+\infty}\frac{u(\tau\xi)}{\tau^{1+2s}} \, d\tau. \]
Notice that 
\[ \int_0^{+\infty}\frac{u(\tau\xi)}{\tau^{1+2s}} \, d\tau = 
\begin{cases}
0 &\text{ if } \langle \xi, e_N \rangle=0 \\
f(\langle \xi, e_N \rangle) &\text{ if } \langle \xi, e_N \rangle \in (0, 1],
\end{cases}\]
where
\[ f(y) = \int_1^{+\infty} \frac{e^{-\tau y}}{\tau^{1+2s}}\, d\tau , \]
which is continuous and monotone decreasing and
\[ \sup_{y \in (0, 1]} f(y)= f(0)=\int_1^{+\infty} \frac{1}{\tau^{1+2s}} \, d\tau. \]
Therefore  we deduce
\[ \I_1^+ u(0) =C_s \int_1^{+\infty} \frac{1}{\tau^{1+2s}} \, d\tau. \]
However, there does not exist any $\xi$ such that $\I_1^+ u(0)=\I_\xi u(0)$.

Let us now consider the case $\I_k^+$ with $2 \le k \le N$. We take into account the function 
\[ u(x)= \begin{cases}
e^{-\langle x, e_N \rangle } &\text{ if } \sum_{i=1}^{N-2} \langle x, e_i \rangle ^2=0, \, \langle x, e_{N-1} \rangle ^2+ \langle x, e_N \rangle ^2>1, \, \langle x, e_N \rangle  >0 \\
0 &\text{ otherwise.}
\end{cases} \]
In this case, 
\[ \I_k^+ u(0)=\sup_{\theta \in  [0, \pi/2] } (\I_{\eta_1} u(0)  + \I_{\eta_2}u(0)), \]
where
\[ \eta_1=(0, \dots, 0, \cos \theta, \sin \theta), \quad \eta_2=(0, \dots, 0, -\sin \theta, \cos \theta). \]
Thus one has
\[ \I_{\eta_1} u(0)  + \I_{\eta_2}u(0)= 
\begin{cases}
\displaystyle{C_s \int_1^{+\infty} \frac{e^{-\tau \sin \theta} + e^{-\tau \cos \theta}}{\tau^{1+2s}} \,d \tau} &\text{ if } \theta \in (0, \pi/2)\\
&\\
\displaystyle{C_s \int_1^{+\infty} \frac{e^{-\tau}}{\tau^{1+2s}} \, d\tau} &\text{ if } \theta=0 \text{ or } \theta=\pi/2.
\end{cases} \]

Now, let us compute the supremum of the function
\[ F(\theta)=\int_1^{+\infty} \frac{e^{-\tau \sin \theta} + e^{-\tau \cos \theta}}{\tau^{1+2s}} \,d \tau= \int_1^{+\infty} \frac{f(\tau, \theta)}{\tau^{1+2s}} \, d\tau. \]

Observe that 
\begin{equation}\label{primo} 
0 \le \frac{f(\tau, \theta)}{\tau^{1+2s}} \le \frac{2}{\tau^{1+2s}} \in L^1(1, +\infty), 
\end{equation}
and that
\begin{equation}\label{secondo}
\frac{1}{\tau^{1+2s}} \abs{\frac{\partial f}{\partial \theta}} = \frac{1}{\tau^{2s}} \abs{-e^{-\tau \sin \theta} \cos \theta + e^{-\tau \cos \theta} \sin \theta } \le \frac{2}{\tau^{2s}} \in L^1(1, +\infty),
\end{equation}
as $s > 1/2$. By \eqref{primo} and \eqref{secondo}, $F(\theta) \in C^1(0, \pi/2)$ and 
\[ F'(\theta)=\int_1^{+\infty} \frac{\frac{\partial f}{\partial \theta}}{\tau^{1+2s}} \, d\tau. \]
Moreover, 
\begin{equation}\label{terzo}
\frac{\partial^2 f}{\partial \theta^2} = \tau^2 e^{-\tau \sin\theta} \cos^2 \theta + \tau e^{-\tau \sin \theta} \sin \theta + \tau^2 e^{-\tau \cos \theta} \sin^2 \theta + \tau e^{-\tau \cos \theta} \cos \theta >0 
\end{equation}
for all $\tau >1$ and $\theta \in (0, \pi/2)$. Also,
\begin{equation}\label{quarto}
\frac{\partial f}{\partial \theta}(\tau, \pi/4)=0. 
\end{equation}
Combining \eqref{terzo} and \eqref{quarto}, we conclude
\[ F'(\theta)<0, \text{ if } \theta \in (0, \pi/4), \quad F'(\theta)>0, \text{ if } \theta \in (\pi/4, \pi/2). \]
Finally, 
\[ \lim_{\theta \to 0^+} F(\theta)=\lim_{\theta \to \pi/2^-} F(\theta)=\int_1^{+\infty} \frac{1+e^{-\tau}}{\tau^{1+2s}} \, d\tau, \]
which implies
\[ \sup_{0 < \theta < \pi/2} F(\theta)=\int_1^{+\infty} \frac{1+e^{-\tau}}{\tau^{1+2s}} \, d\tau. \]

Therefore, 
\[ \I_k^+ u(0)= C_s \int_1^{+\infty} \frac{1+e^{-\tau}}{\tau^{1+2s}} \, d\tau \]
however there does not exists $\theta \in [0, \pi/2]$ such that 
\[ \I_{\eta_1} u(0)  + \I_{\eta_2}u(0)= C_s \int_1^{+\infty} \frac{1+e^{-\tau}}{\tau^{1+2s}} \, d\tau. \]

\begin{proposition}\label{semicont}
Let $u\in C^2(\Omega)\cap L^\infty(\R^N)$ and consider the maps
\begin{equation*}
\begin{split}
&\Psi:(x,\xi)\in\Omega\times{\mathcal S}^{N-1}\mapsto \I_\xi u(x)\\
&\I^\pm_ku:x\in\Omega\mapsto \I^\pm_ku(x).
\end{split}
\end{equation*}
If $u\in LSC(\R^N)$ (respectively $USC(\R^N)$, $C(\R^N)$) then
\begin{enumerate}
	\item [(i)]  $\Psi\in LSC(\Omega\times{\mathcal S}^{N-1})$ (respectively $USC(\Omega\times{\mathcal S}^{N-1})$, $C(\Omega\times{\mathcal S}^{N-1})$);
	\item [(ii)] $\I^\pm_ku\in LSC(\Omega)$ (respectively $USC(\Omega)$, $C(\Omega)$).
\end{enumerate}
\end{proposition}
\begin{proof}
(i) Let  $(x_n,\xi_n)\to (x_0,\xi_0)\in\Omega\times{\mathcal S}^{N-1}$ as $n\to+\infty$. Fix $R>0$ such that $\overline B_R(x_0)\subset \Omega$ and set $M=\displaystyle\max_{x\in\overline B_R(x_0)}\left\|D^2u(x)\right\|$.
For $\rho\in(0,\frac R2)$ it holds that $B_{2\rho}(x_0)\subset B_R(x_0)$ and, for $n$ sufficiently large and any $\tau\in[0,\rho)$, that $x_n\pm\tau\xi_n\in B_{2\rho}(x_0)$. By  a second order Taylor expansion we have
\[
\I_{\xi_n}u(x_n)-\I_{\xi_0}u(x_0)\geq-\frac{M\rho^{2-2s}}{1-s}+C_s \int_\rho^{+\infty}\frac{\delta(u, x_n, \tau \xi_n)}{\tau^{1+2s}}\,d\tau -C_s \int_\rho^{+\infty}\frac{\delta(u, x_0, \tau \xi_0)}{\tau^{1+2s}}\,d\tau\,.
\]
Since $u(x_n)\to u(x_0)$ as $n\to+\infty$, because of the continuity of $u$ in $\Omega$, then using the lower semicontinuity of $u$ in $\R^N$ we have
$$
\liminf_{n\to+\infty}\delta(u, x_n, \tau \xi_n) \geq \delta(u, x_0, \tau \xi_0)
$$ 
for any $\tau\in(0,+\infty)$.  Moreover, taking into account that $\rho>0$ and  $u\in L^\infty(\R^N)$, by means of Fatou's lemma we also infer that 
$$
\liminf_{n\to+\infty}[ \I_{\xi_n}u(x_n)-\I_{\xi_0}u(x_0)] \geq-\frac{M\rho^{2-2s}}{1-s}.
$$
Since $\rho$ can be chosen arbitrarily small we conclude that 
$$
\liminf_{n\to+\infty}\Psi(x_n, \xi_n)\geq \Psi(x_0, \xi_0).
$$
In a similar way one can prove that $\Psi\in USC(\Omega\times{\mathcal S}^{N-1})$ if $u\in USC(\R^N)$. In particular $\Psi\in C(\Omega\times{\mathcal S}^{N-1})$ when $u$ is continuous in $\R^N$.

\medskip

(ii) By the assumption $u\in C^2(\Omega)\cap L^\infty(\R^N)$, we first note that, for any $x\in\Omega$,  $\I_\xi u(x)$ is uniformly bounded with respect to $\xi\in\mathcal S^{N-1}$. Hence
$$
-\infty<\I^-_ku(x)\leq\I^+_ku(x)<+\infty.
$$  
Moreover, for any compact $K \subset \Omega$ there exists a constant $M_K$ such that 
\[ -M_K \le \I_k^- u \le \I_k^+ u \le M_K. \]
Henceforth we shall consider $\I^-_k$, the other case being similar. 

Let $x_n\to x_0\in\Omega$ as $n\to+\infty$  and let $\varepsilon>0$. By the definitions of lower limit and $\I_k^-u$, there exist a subsequence $(x_{n_m})_{m}$ and $k$ sequences $(\xi_i(m))_m\subset\mathcal{S}^{N-1}$, $i=1,\ldots,k$, such that for any $m \in \N$ 
\begin{equation}\label{2505eq1}
\liminf_{n\to+\infty}\I_k^-u(x_n)+2\varepsilon\geq\I_k^-u(x_{n_m})+\varepsilon\geq\sum_{i=1}^k\Psi(x_{n_m},\xi_i(m)).
\end{equation}
Up to extract a further subsequence, we can assume that $\xi_i(m)\to\bar\xi_i$, as $m\to+\infty$, for any $i=1,\ldots,k$. Since $\Psi\in LSC(\Omega\times{\mathcal S}^{N-1})$ by i),  we can pass to the limit as $m\to+\infty$ in \eqref{2505eq1} to get
\begin{equation*}
\liminf_{n\to+\infty}\I_k^-u(x_n)+2\varepsilon\geq\sum_{i=1}^k\Psi(x_0,\bar\xi_i)\geq\I^-_ku(x_0).
\end{equation*}
This implies that $\I_k^-u(x)\in LSC(\Omega)$  sending $\varepsilon\to0$.

The proof that $\I_k^-u(x)\in USC(\Omega)$ under the assumption $u\in USC(\R^N)$ is more standard since $\I_k^-u(x)=\inf_{\left\{\xi_i\right\}_{i=1}^k\in{\mathcal V}_k}\sum_{i=1}^k\Psi(x,\xi_i)$ and $\Psi(x,\xi_i)\in USC(\Omega)$ by i).

Lastly if $u\in C(\R^N)$, by the previous cases $\I_k^-$ is in turn a continuous function in $\Omega$.
\end{proof}

\section{Comparison and maximum principles}\label{sec:comparison}

We consider the problems
\begin{equation}\label{2505eq2}
\left\{\begin{array}{cl}
\I^\pm_ku+c(x)u=f(x) &  \text{in $\Omega$}\\
u=0 & \text{in $\R^N\backslash\Omega$}
\end{array}\right.
\end{equation}

\begin{theorem}\label{comparison}
Let $\Omega\subset\R^N$ be a bounded domain and let $c(x),f(x)\in C(\Omega)$ be such that $\left\|c^+\right\|_{\infty}<C_s  \frac ks(\text{diam}(\Omega))^{-2s}$. If $u\in USC(\overline\Omega)\cap L^\infty(\R^N)$ and $v\in LSC(\overline\Omega)\cap L^\infty(\R^N)$ are respectively sub and supersolution of \eqref{2505eq2}, then $u\leq v$ in  $\Omega$.
\end{theorem}
\begin{proof}
We shall detail the proof in the case $\I^+_k$, the same arguments applying to $\I^-_k$ as well.
We argue by contradiction by supposing that there exists $x_0\in\Omega$ such that
$$
\max_{\R^N}(u-v)=u(x_0)-v(x_0)>0.
$$
Doubling the variables,  for $n\in\mathbb N$ we consider  $(x_n,y_n)\in\overline\Omega\times\overline\Omega$ such that 
\begin{equation}\label{2505eq3}
\max_{\overline\Omega\times\overline\Omega}(u(x)-v(y)-n|x-y|^2)=u(x_n)-v(y_n)-n|x_n-y_n|^2\geq u(x_0)-v(x_0).
\end{equation}
Using \cite[Lemma 3.1]{CIL}, up to subsequences, we have
\begin{equation}\label{2505eq6}
\lim_{n\to+\infty}(x_n,y_n)=(\bar x,\bar x)\in\Omega\times\Omega
\end{equation}
and 
\begin{equation}\label{2505eq7}
\lim_{n\to+\infty}u(x_n)=u(\bar x),\quad \lim_{n\to+\infty}v(x_n)=v(\bar x),\quad u(\bar x)-v(\bar x)=u(x_0)-v(x_0).
\end{equation}
By semicontinuity of $u$ and $v$ we can find moreover $\varepsilon>0$ such that 
\begin{equation}\label{2505eq4}
u(x)<u(x_0)-v(x_0) \quad\forall x\in\Omega_\varepsilon
\end{equation}
and also
\begin{equation}\label{2505eq4.1}
-v(x)<u(x_0)-v(x_0) \quad\forall x\in\Omega_\varepsilon
\end{equation}
where $\Omega_\varepsilon=\left\{x\in\overline\Omega:\;\text{dist}(x,\partial\Omega)<\varepsilon\right\}$.
We first claim that for $n\geq\frac{\left\|u\right\|_\infty+\left\|v\right\|_\infty}{\varepsilon^2}$ 
\begin{equation}\label{2505eq5}
\max_{\overline\Omega\times\overline\Omega}[ u(x)-v(y)-n|x-y|^2]=\max_{\R^N\times\R^N}[u(x)-v(y)-n|x-y|^2]\,.
\end{equation}
To show \eqref{2505eq5} take any $(x,y)\notin\overline\Omega\times\overline\Omega$:

Case 1. If $|x-y|\geq\varepsilon$, then 
$u(x)-v(y)-n|x-y|^2\leq\left\|u\right\|_\infty+\left\|v\right\|_\infty-n\varepsilon^2\leq0$; 

Case 2. If $|x-y|<\varepsilon$ and both $x\notin\overline\Omega$ and $y\notin\overline\Omega$, then
	$u(x)-v(y)-n|x-y|^2\leq0$;
	
Case 3. If $|x-y|<\varepsilon$ and  $x\notin\overline\Omega,\ y\in\overline\Omega$ or $x\in\overline\Omega,\ y\notin\overline\Omega$, then using \eqref{2505eq4} and \eqref{2505eq4.1} we infer that $u(x)-v(y)-n|x-y|^2<u(x_0)-v(x_0)$.

Thus, \eqref{2505eq5} is proved. 

 Taking $\varphi_n(x):=u(x_n)+ n|x-y_n|^2 - n|x_n-y_n|^2$ and $\phi_n(y)=v(y_n)- n|x_n-y|^2 +n|x_n-y_n|^2$, we see that $\varphi_n$ touches $u$ in $x_n$ from above, while $\phi_n$ touches $v$ in $y_n$ from below. Hence for any $\rho>0$
\begin{equation}\label{2505eq8}
\begin{split}
f(x_n)&\leq c(x_n)u(x_n)+C_s \sup_{\left\{\xi_i\right\}_{i=1}^k\in{\mathcal V}_k}\sum_{i=1}^k \Bigg(\int_0^\rho\frac{\delta(\varphi_n, x_n, \tau \xi_i)}{\tau^{1+2s}}\,d\tau 
+\int_\rho^{+\infty}\frac{\delta(u, x_n, \tau \xi_i)}{\tau^{1+2s}}\,d\tau\Bigg)\\
&=c(x_n)u(x_n)+\frac{kn\rho^{2-2s}}{1-s}+C_s \sup_{\left\{\xi_i\right\}_{i=1}^k\in{\mathcal V}_k}\Bigg(\sum_{i=1}^k\int_\rho^{+\infty}\frac{\delta(u, x_n, \tau \xi_i)}{\tau^{1+2s}}\,d\tau\Bigg).
\end{split}
\end{equation}
In a dual fashion
\begin{equation}\label{2505eq9}
f(y_n)\geq c(y_n)v(y_n)-\frac{kn\rho^{2-2s}}{1-s}+C_s \sup_{\left\{\xi_i\right\}_{i=1}^k\in{\mathcal V}_k}\Bigg(\sum_{i=1}^k\int_\rho^{+\infty}\frac{\delta(v, y_n, \tau \xi_i)}{\tau^{1+2s}}\,d\tau\Bigg).
\end{equation}
Subtracting \eqref{2505eq8} and \eqref{2505eq9} we then obtain
\begin{equation}\label{2505eq10}
\begin{split}
f(x_n)-f(y_n)&\leq \frac{2kn\rho^{2-2s}}{1-s}+c(x_n)u(x_n)-c(y_n)v(y_n)
\\ &\hspace{1cm} +C_s  \sup_{\left\{\xi_i\right\}_{i=1}^k\in{\mathcal V}_k}\sum_{i=1}^k\int_\rho^{+\infty}\frac{\delta(u, x_n, \tau\xi_i) - \delta(v, y_n, \tau \xi_i)}{\tau^{1+2s}}\,d\tau.
\end{split}
\end{equation}
From \eqref{2505eq3} and \eqref{2505eq5}  we have
$$
u(x)-v(y)-n|x-y|^2\leq u(x_n)-v(y_n)-n|x_n-y_n|^2\quad\forall x,y\in\R^N.
$$
Choosing in particular $x=x_n\pm\tau \xi_i$ and $y=y_n\pm\tau\xi_i$ we deduce that 
\[ \delta(u, x_n, \tau\xi_i) - \delta(v, y_n, \tau \xi_i) \le 0 \]
 for any $\tau>0$ and for any $|\xi_i|=1$. Thus \eqref{2505eq10} implies, assuming without loss of generality that $\rho < \text{diam}(\Omega)$, 
\begin{equation}\label{2505eq11}
\begin{split}
f(x_n)-f(y_n)&\leq \frac{2kn\rho^{2-2s}}{1-s}+c(x_n)u(x_n)-c(y_n)v(y_n)\\
&\hspace{0.5cm}+C_s  \sup_{\left\{\xi_i\right\}_{i=1}^k\in{\mathcal V}_k}\sum_{i=1}^k\int_{\text{diam}(\Omega)}^{+\infty}\frac{\delta(u, x_n, \tau\xi_i) - \delta(v, y_n, \tau \xi_i)}{\tau^{1+2s}}\,d\tau.
\end{split}
\end{equation}
Since $\Omega\subset B_{\text{diam}(\Omega)}(x_n)$ and $x_n\pm\tau\xi_i \notin B_{\text{diam}(\Omega)}(x_n)$ for any $\tau\geq\text{diam}(\Omega)$, then $u(x_n\pm\tau\xi_i)\leq0$. For the same reason $v(y_n\pm\tau\xi_i)\ge0$ when $\tau\geq\text{diam}(\Omega)$. Hence 
$$
\delta(u, x_n, \tau\xi_i) - \delta(v, y_n, \tau \xi_i)\leq -2(u(x_n)-v(y_n))
$$ 
and 
\begin{equation}\label{2505eq12}
\begin{split}
f(x_n)-f(y_n)\leq & \; \frac{2kn\rho^{2-2s}}{1-s}+c(x_n)u(x_n)-c(y_n)v(y_n)\\
&\quad -C_s  (u(x_n)-v(y_n))\frac ks(\text{diam}(\Omega))^{-2s}.
\end{split}
\end{equation}
Letting first $\rho\to0$, then $n\to+\infty$ and using \eqref{2505eq6}-\eqref{2505eq7} we obtain
$$
0\leq (u(x_0)-v(x_0))\left(c(\bar x)-C_s  \frac ks(\text{diam}(\Omega))^{-2s}\right)
$$
which is a contradiction since $u(x_0)-v(x_0)>0$ and $\left\|c^+\right\|_{\infty}<C_s  \frac ks(\text{diam}(\Omega))^{-2s}$.
\end{proof}

In what follows, we clarify what we mean by (weak) maximum/minimum principle. 
\begin{definition}
We say that the operator $\mathcal{I}$ satisfies the weak maximum principle in $\Omega$ if 
\[ \mathcal{I} u  \ge 0 \text{ in } \Omega, \quad u \le 0 \text{ in } \R^N \setminus \Omega \quad \Longrightarrow \quad u \le 0 \text{ in } \Omega, \]
and that it satisfies the strong maximum principle in $\Omega$ if 
\[ \mathcal{I} u  \ge 0 \text{ in } \Omega, \quad u \le 0 \text{ in }\R^N \quad \Longrightarrow \quad u < 0 \text{ or } u \equiv 0 \text{ in } \Omega. \]
Correspondingly, $\mathcal{I}$ satisfies the weak minimum principle in $\Omega$ if 
\[ \mathcal{I} u  \le 0 \text{ in } \Omega, \quad u \ge 0 \text{ in } \R^N \setminus \Omega \quad \Longrightarrow \quad u \ge 0 \text{ in } \Omega, \]
and it satisfies the strong minimum principle in $\Omega$ if 
\[ \mathcal{I} u  \le 0 \text{ in } \Omega, \quad u \ge 0 \text{ in } \R^N \quad \Longrightarrow \quad u > 0 \text{ or } u \equiv 0 \text{ in } \Omega. \]
\end{definition}

The weak minimum/maximum principle follows by applying the comparison principle Theorem \ref{comparison} with $v=0$ or $u=0$. However, the operators $\I_k^\pm$ do not always satisfy the strong maximum or minimum principle, see also \cite{BGT}. 
\begin{theorem}\label{SMP}
The following conclusions hold.
\begin{enumerate}
\item[(i)] The operators $\I_k^-$, with $k < N$, do not satisfy the strong minimum principle in $\Omega$.
\item[(ii)] The operator $\I_N^-$ satisfies the strong minimum principle in $\Omega$. 
\item[(iii)] The operators $\I_k^+$, with $k \le N$,  satisfy the strong minimum principle in $\Omega$. 
\end{enumerate}
\end{theorem}

\begin{remark}
We notice that since $\I_k^+ (-u)= -\I_k^- u$, corresponding results hold for the maximum principle. 
\end{remark}

\begin{proof}
(i) We refer to Proposition 2.2 in \cite{BGT} for a counterexample. 

(ii)
Let us assume that $u$ satisfies
\[ \begin{cases}
\I_N^- u  \le 0 &\text{ in } \Omega \\
u \ge 0 &\text{ in } \R^N
\end{cases} \] 
and let $u(x_0)=0$ for some $x_0 \in \Omega$. We want to prove that $u \equiv 0$ in $\Omega$. Let us proceed by contradiction, and assume there exists $y \in \Omega$ such that $u(y) >0$. Let us choose a ball $B_R(y)$ such that 
\begin{itemize}
\item $B_R(y) \subset \Omega$
\item $u(x) >0$ for all $x \in \overline B_R(y) \setminus \{ x_1 \}$
\item there exists $x_1 \in \partial B_R(y)$ such that $u(x_1)=0$. 
\end{itemize}
Then, by definition of viscosity super solutions, for fixed $\rho>0$ and $\varphi \in C^2(B_\rho(x_1))$, for which $x_1$ is a minimum point for $u-\varphi$, and for every $\varepsilon>0$, there exists a orthonormal basis $ \{ \xi_1, \dots, \xi_N \}=\{ \xi_1(\varepsilon), \dots, \xi_N(\varepsilon) \}$ such that 
\begin{equation}\label{eqn:IN strong} \varepsilon  \ge C_s  \sum_{i=1}^N \Big( \int_{0}^\rho \frac{\delta(\varphi, x_1, \tau \xi_i)}{\tau^{1+2s}} \, d\tau 
+\int_{\rho}^{+\infty} \frac{\delta(u, x_1, \tau \xi_i)}{\tau^{1+2s}}  \, d\tau \Big). 
\end{equation}
Fix $\rho < \frac{2R}{\sqrt{N}}$, and choose $\varphi \equiv 0$ on $B_\rho(x_1)$. Moreover, we know that there exists $j=j(\varepsilon)$ such that 
\[ \langle \xi_j, \widehat{x_1- y} \rangle \ge \frac{1}{\sqrt{N}}, \quad \text{ with } \widehat{x_1- y} =\frac{x_1-y}{\abs{x_1-y}}. \]  
In particular, one has $\rho < 2R\langle \xi_j, \widehat{x_1- y} \rangle$. 
Then, taking into account that $u(x_1)=0$ and $u \ge 0$, from \eqref{eqn:IN strong} one has
\begin{align*} \varepsilon &\ge C_s \sum_{i=1}^N \int_{\rho}^{+\infty} \frac{u(x_1 + \tau \xi_i)+u(x_1-\tau \xi_i)}{ \tau^{1+2s}} \, d\tau \\
&= C_s \sum_{i \ne j} \int_{\rho}^{+\infty} \frac{u(x_1 + \tau \xi_i)+ u(x_1-\tau \xi_i) }{\tau^{1+2s}} \, d\tau + C_s  \int_{\rho}^{+\infty} \frac{u(x_1 + \tau \xi_j)+ u(x_1-\tau \xi_j)}{ \tau^{1+2s}} \, d\tau \\
& \ge C_s  \int_{\rho}^{+\infty} \frac{u(x_1 - \tau \xi_j)}{\tau^{1+2s}} \, d\tau \ge C_s  \int_{\rho}^{2R\langle \xi_j, \widehat{x_1- y} \rangle} \frac{u(x_1 - \tau \xi_j)}{\tau^{1+2s}} \, d\tau \\
&\ge C_s  \frac1{2s} \left( \rho^{-2s} - \left(\frac{2R}{\sqrt N} \right)^{-2s} \right) \min_{\overline B_R(y) \setminus B_\rho(x_1)} u, 
\end{align*}
as $x_1-\tau \xi_j \in \overline B_R(y) \setminus B_\rho(x_1)$ if $\rho<\tau < 2R\langle \xi_j, \widehat{x_1- y} \rangle$, which gives the contradiction if $\varepsilon$ is small enough.  

\textit{(iii)}
The conclusion for the operators $\I_k^+$ follows recalling 
\[ \I_k ^+ u(x) \le 0\;  \Rightarrow \; \I_N^- u(x)\le 0. \]
Indeed, since $\I_k ^+ u(x) \le 0$ one has $\sum_{i=1}^k \I_{\xi_i}  u(x) \le 0$ for any $\{ \xi_1, \dots, \xi_k\} \in \mathcal{V}_k$. Fix any $\{ \bar \xi_1, \dots, \bar \xi_N \} \in \mathcal{V}_N$, and denote with $\mathcal{A}_k$ the set of all subsets of cardinality $k$ of $\{ \bar \xi_1, \dots, \bar \xi_N \}$. Clearly, $\mathcal{A}_k \subset \mathcal{V}_k$. In particular,
\[ 0 \ge \sum_{\{ \xi_i \} \in \mathcal{A}_k } \sum_{i=1}^k \I_{\xi_i}  u(x) ={{N-1}\choose {k-1}} \sum_{i=1}^N \I_{\bar \xi_i}  u(x), \]
from which the conclusion. 
\end{proof}

\begin{remark}
Notice that the proofs above only require $\Omega$ to be connected, and not necessarily bounded. 
\end{remark}

\begin{remark}\label{rmk:ordine}
The same proof as in item \textit{(iii)} shows that 
\[ \I_k ^+ u(x) \le 0\;  \Rightarrow \; \I_{k+1}^+ u(x)\le  \I_k ^+ u(x) \]
and
\[ \I_k ^- u(x) \le 0\;  \Rightarrow \; \I_{k-1}^- u(x)\le  \I_k ^- u(x)\,. \]
\end{remark}

Actually, the operators $\I_k^+$ satisfy a stronger condition than the strong minimum principle, which is also satisfied by the fractional Laplacian, and which turns out to be false for $\I_N^-$. 
\begin{proposition}\label{prop:strong}
One has 
\begin{enumerate}
\item[(i)] The operators $\I_k^+$, with $k \le N$,  satisfy the following 
\[ \I_k^+ u(x) \le 0 \text{ in } \Omega, \quad u \ge 0 \text{ in }\R^N \; \Rightarrow \;
u > 0 \text{ in } \Omega \text{ or } u \equiv 0 \text{ in } \R^N. \]
\item[(ii)] There exist functions $u$ such that $\I_N^- u \le 0$ in $ \Omega$, $u \equiv 0$ in $\overline \Omega$, and $u \not\equiv 0$ in $\R^N \setminus \overline \Omega$. 
\end{enumerate}
\end{proposition}

\begin{figure}
\centering
\includegraphics[width=0.7\textwidth]{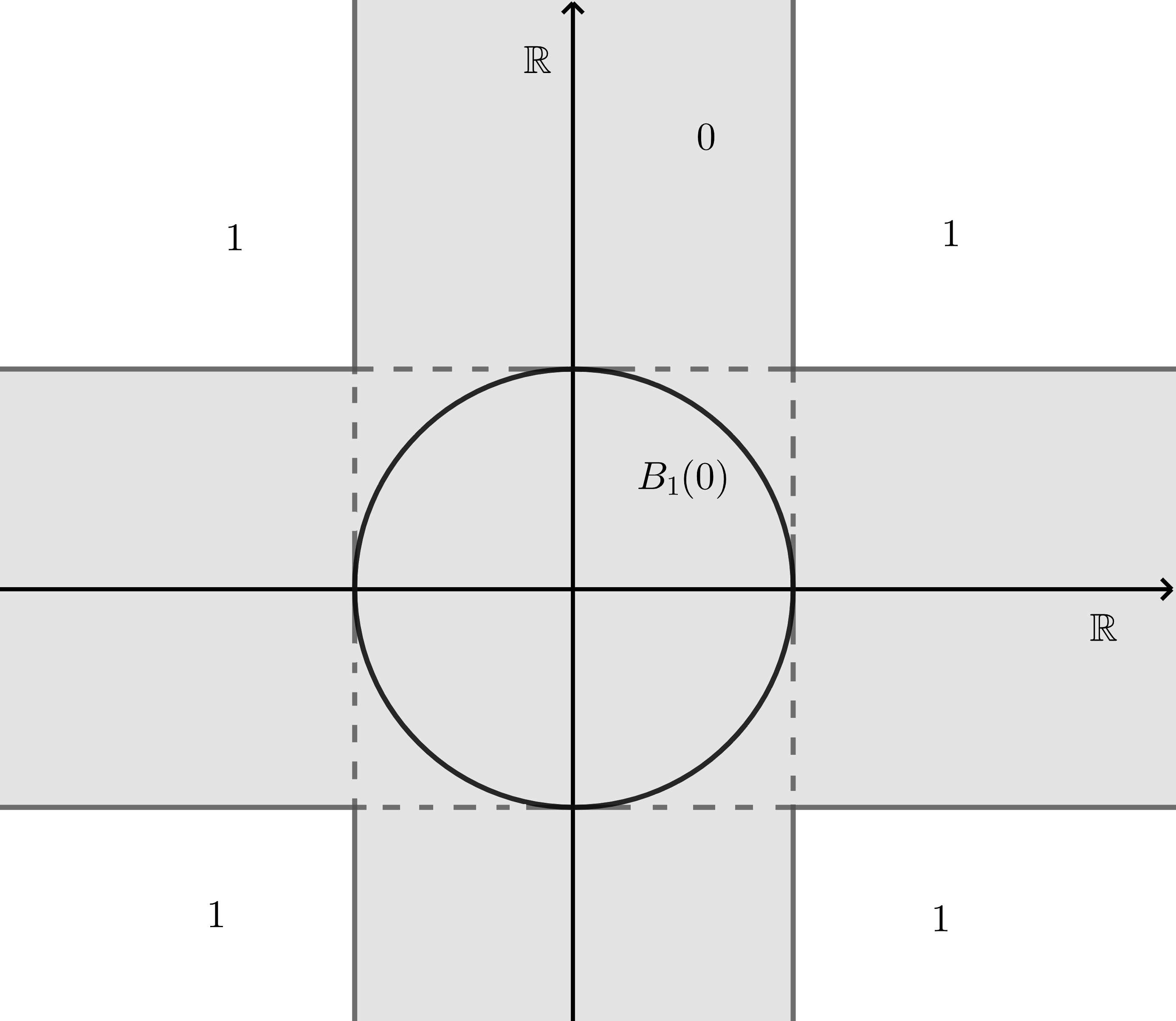}
\caption{Graphic representation of the function $u$ in the proof of Proposition \ref{prop:strong} (ii), with $N=2$.}
\label{croce}
\end{figure}

\begin{proof}
(i) Take $u$ which satisfies the assumptions of the minimum principle, and assume there exists $x_0 \in \Omega$ such that $u(x_0)=0$. By the strong minimum principle in $\Omega$, we know that $u \equiv 0$ in $\Omega$, in particular $u \ge 0$ in $\R^N$. Choose any orthonormal basis of $\R^N$ $\{ \xi_1, \dots, \xi_{N} \}$. Thus, recalling that $u\ge0$ in $\R^N$ 
\begin{align*} 0 \ge \I_k^+ u(x_0) & \ge \sum_{i=1}^k \mathcal{I}_{\xi_i} u(x_0) \\
& =C_s \sum_{i=1}^k \int_0^{+\infty} \frac{u(x_0 + \tau \xi_i) + u(x_0-\tau \xi_i)}{\tau^{1+2s}} \, d\tau.
\end{align*}
Hence, since $u \ge 0$ in $\R^N$, we conclude that $u \equiv 0$ on every line with direction $\xi_i$, and passing by $x_0$. Since the directions are arbitrary, we get the conclusion. 

(ii) Take 
\[ u(x)=\begin{cases}
0  &\text{ if there exists $i =1, \dots, N$ such that } \abs{\langle x, e_i \rangle } \le 1\\
1 &\text{ otherwise,} 
\end{cases} \]
see also Figure \ref{croce}, and notice that 
\[ \I_N^- u (x)\le \sum_{i=1}^N \mathcal{I}_{e_i} u (x) =0 \text{ in } B_1(0), \]
where $e_i$ is the canonical basis. 
Moreover, $u \equiv 0$ in $\overline B_1(0)$, however $u \not\equiv 0$ in $\R^N \setminus \overline B_1(0)$. 
\end{proof}

We now prove a Hopf-type Lemma. 
We will borrow some ideas from \cite{GrecoServadei}, where the fractional Laplacian is taken into account. 
The next known computation provides a useful barrier function. 
\begin{lemma}[Section 3.6 in \cite{BV}]\label{funzione barriera}
For any $\xi \in \mathcal{S}^{N-1}$ one has 
\[ \mathcal{I}_\xi (R^2-\abs{x}^2)^s_+= - C_s \beta(1-s, s)  \, \text{ in } B_R(0), \]
where 
\[ \beta(1-s, s)=\int_0^1 t^{-s} (1-t)^{s-1} \, dt \]
is the Beta function. In particular, 
\[ \I_k^+ (R^2-\abs{x}^2)^s_+= \I_k^- (R^2-\abs{x}^2)^s_+= - k \, C_s \beta(1-s, s) \, \text{ in } B_R(0). \]
\end{lemma}

\begin{proposition}\label{hopf}
Let $\Omega$ be a bounded $C^2$ domain, and let $u$ satisfy  
\[ \begin{cases}
\I_N^- u \le 0 &\text{ in } \Omega \\
u \ge 0 &\text{ in } \R^N \setminus \Omega.
\end{cases}\]
Assume $u \not \equiv 0$ in $\Omega$. Then there exists a positive constant $c=c(\Omega, u)$ such that 
\begin{equation}\label{eq hopf} u(x) \ge c\, d(x)^s\quad\forall x\in\overline\Omega. \end{equation} 
\end{proposition}

Notice that the conclusion is not true for the operators $\I_k^-$, $k < N$. 
Indeed, consider the function
\[ u(x)=\begin{cases}
e^{-\frac{1}{1-\abs{x}^2}} &\text{ if } \abs{x} < 1 \\
0 &\text{ if } \abs{x} \ge 1 \end{cases} \]
and take $\{ \xi_i \} \in \mathcal{V}_k$ such that $\langle x, \xi_i \rangle=0$ for any $i=1, \dots, k$. Hence
\[ \abs{x+\tau \xi_i}^2 =\abs{x}^2 + \tau^2 \ge \abs{x}^2 \]
and using  the radial monotonicity of $u$ 
\[ \I_k^- u(x) \le \sum_{i=1}^k \mathcal{I}_{\xi_i} u(x) \le 0 \text{ in } B_1(0). \]
However, $u$ clearly does not satisfy 
\[ u(x) \ge c\, d(x)^\gamma \]
for any positive constants $c, \gamma$. 

As a consequence of Proposition \ref{hopf}, we immediately obtain the following 
\begin{corollary}\label{cor hopf}
Let $\Omega$ be a bounded $C^2$ domain, and let $u$ satisfy  
\[ \begin{cases}
\I_k^+ u \le 0 &\text{ in } \Omega \\
u \ge 0 &\text{ in } \R^N \setminus \Omega.
\end{cases}\]
Assume $u \not \equiv 0$ in $\Omega$. Then  
\[ u(x) \ge c\, d(x)^s \]
for some positive constant $c=c(\Omega, u)$. 
\end{corollary}

\begin{remark}
We also point out that from Proposition \ref{hopf} one can deduce the strong maximum/minimum principle for the operators $\I_k^+$, $\I_N^-$, which however follows also by a more direct argument as we showed in Theorem \ref{SMP}. 
\end{remark}
\begin{proof}[Proof of Proposition \ref{hopf}]
By the weak and strong minimum principles, see Theorem \ref{comparison} and Theorem \ref{SMP}-(ii), $u>0$ in $\Omega$. Therefore, for any $K$ compact subset of $\Omega$ we have
\begin{equation}\label{bound} \inf_{y \in K} u(y) >0. \end{equation} 
 Without loss of generality we can further assume that $u$ vanishes somewhere in $\partial\Omega$, otherwise the conclusion is obvious.\\
Since $\Omega$ is a $C^2$ domain, there exists a positive constant $\varepsilon$, depending on $\Omega$, such that for any $x \in \Omega_\varepsilon=\{ x \in \Omega: d(x) < \varepsilon \}$ there are a unique $z \in \partial \Omega$ for which $d(x)=\abs{x-z}$ and a ball $B_{2\varepsilon}(\bar y) \subset \Omega$ such that $\overline{B_{2\varepsilon}(\bar y)} \cap (\R^N \setminus \Omega)=\{ z \}$.  \\
Now we consider the radial function $w(x)={((2\varepsilon)^2-\abs{x- \bar y}^2)}^s_+$ which satisfies, see  Lemma \ref{funzione barriera}, the equation 
\[ \I_N^- w=-N \, C_s \beta(1-s, s) \, \text{ in } B_{2\varepsilon}(\bar y). \]
We claim that there exists $\bar n= \bar n(u, \varepsilon)$ such that 
\[ u \ge w_{\bar n} \text{ in } \R^N, \]
where 
\[ w_n(x) = \frac 1n w(x). \]
This implies \eqref{eq hopf}. Indeed, for any $x\in\Omega_\varepsilon$
\begin{equation}\label{2406eq1} w_{\bar n} (x)=\frac 1{\bar n}  ((2\varepsilon)^2-\abs{x- \bar y}^2)^s_+ \ge \frac {2\varepsilon}{\bar n} \abs{x-z}^s =\frac {2\varepsilon}{\bar n} d(x)^s, \end{equation}
and
 \begin{equation}\label{2406eq2}
u(x)\geq \min_{y\in\Omega\backslash\Omega_\varepsilon}\frac{u(y)}{d(y)^s}d(x)^s\quad\forall x\in \Omega\backslash\Omega_\varepsilon.
\end{equation}
 From \eqref{2406eq1}-\eqref{2406eq2} we obtain \eqref{eq hopf} with $c=\min\left\{\frac {2\varepsilon}{\bar n} ,\min_{y\in\Omega\backslash\Omega_\varepsilon}\frac{u(y)}{d(y)^s}\right\}$.

We proceed by contradiction in order to prove the claim, hence, we suppose that for any $n \in \mathbb{N}$ 
\[ v_n = w_n-u \]
is USC and positive somewhere. From now on, for simplicity of notation, we assume that $B_{2\varepsilon}(\bar y) =B_1(0)$. 
Since
\[ w_n =0 \le u \text{ in } \R^N \setminus B_1(0), \]
we know that it attains its positive maximum  $x_n$ in $B_1(0) \subset \Omega$. One has 
\[ 0 <u(x_n) < w_n (x_n). \]
Also, $w_n \to 0$ uniformly in $\R^N$, thus
\begin{equation}\label{2406eq3} \lim_{n \to +\infty} u(x_n) =0. \end{equation}
Therefore, recalling \eqref{bound}, $ \abs{x_n} \to 1 $
as $n \to \infty$, hence in particular $x_n \in B_1(0) \setminus B_{r_0}(0)$, where $r_0=\sqrt{1 - \frac 1{2N}}$, and $d(x_n) < (1-r_0)/2$ for $n$ large enough.

Since $\I_N^- u \le 0$ in $\Omega$,  we know that for every test function $\varphi \in C^2(B_\rho(x_n))$ such that $x_n$ is a minimum point to $u-\varphi$, one has
\[ \inf_{\{ \xi_i \} \in \mathcal{V}_N} \sum_{i=1}^N \left( \int_{0}^{\rho} \frac{\delta(\varphi, x_n, \tau \xi_i)}{\tau^{1+2s}} \, d\tau  + \int_{\rho}^{+\infty} \frac{\delta(u, x_n, \tau \xi_i)}{\tau^{1+2s}} \, d\tau  \right) \le 0, \]
and in particular for any $n \in \N$ there exists $\{ \xi_1(n), \dots, \xi_N(n) \}$ orthonormal basis of $\R^N$ such that 
\begin{equation}\label{u supersol hopf} \sum_{i=1}^N \left( \int_{0}^{\rho} \frac{\delta(\varphi, x_n, \tau \xi_i(n))}{\tau^{1+2s}} \, d\tau  + \int_{\rho}^{+\infty}\frac{\delta(u, x_n, \tau \xi_i(n))}{\tau^{1+2s}} \, d\tau  \right) \le \frac 1n. \end{equation}
Since $\{ \xi_1(n), \dots, \xi_N(n) \}$ is a basis of $\R^N$, then there exists at least one $\xi_i(n)$ such that  $\langle \hat x_n, \xi_i(n) \rangle \ge \frac{1}{\sqrt N}$. Without loss of generality we can suppose that $\xi_i(n)=\xi_1(n)$.
Let us choose $\rho = d(x_n) < (1- r_0)/2$, and $\varphi(x)=w_n (x) \in C^2(B_\rho(x_n))$ as test function. 

We consider the left hand side of \eqref{u supersol hopf}, and we aim at providing a positive lower bound independent on $n$, which will give the desired contradiction. 
Let us start with the second integral in \eqref{u supersol hopf} for each fixed $i=2, \dots, N$, and let us notice that since $x_n$ is a maximum point for $v_n$ 
\[ \int_{\rho}^{+\infty}\frac{\delta(u, x_n, \tau \xi_i(n))}{\tau^{1+2s}}\, d\tau \ge \int_{\rho}^{+\infty}\frac{\delta(w_n, x_n, \tau \xi_i(n))}{\tau^{1+2s}}\, d\tau. \]
On the other hand, in order to estimate the integral for $i=1$, we split it as follows:
\begin{equation}\label{int2} \int_{\rho}^{+\infty}\frac{\delta(u, x_n, \tau \xi_1(n))}{\tau^{1+2s}}\, d\tau=J_1+ J_2+J_3, \end{equation}
where  
\[ J_1= \int_{\rho}^{\tau_1(n)}  \frac{\delta(u, x_n, \tau \xi_1(n))}{\tau^{1+2s}} \, d \tau, \]
\[ J_2= \int_{\tau_1(n)}^{\tau_2(n)}  \frac{\delta(u, x_n, \tau \xi_1(n))}{\tau^{1+2s}} d \tau\]
and  
\[ J_3= \int_{\tau_2(n)}^{+\infty}  \frac{\delta(u, x_n, \tau \xi_1(n))}{\tau^{1+2s}}d \tau, \]
with  
\[ \tau_1(n)=\frac{\abs{x_n}}{\sqrt N} - \sqrt{ 1 - \frac{1}{2N} - \abs{x_n}^2\left(1-\frac 1N\right)} \]
and 
\[ \tau_2(n)= \frac{\abs{x_n}}{\sqrt N} + \sqrt{ 1 - \frac{1}{2N} - \abs{x_n}^2\left(1-\frac 1N\right)}. \]

\begin{figure}
\centering
\includegraphics[width=0.5\textwidth]{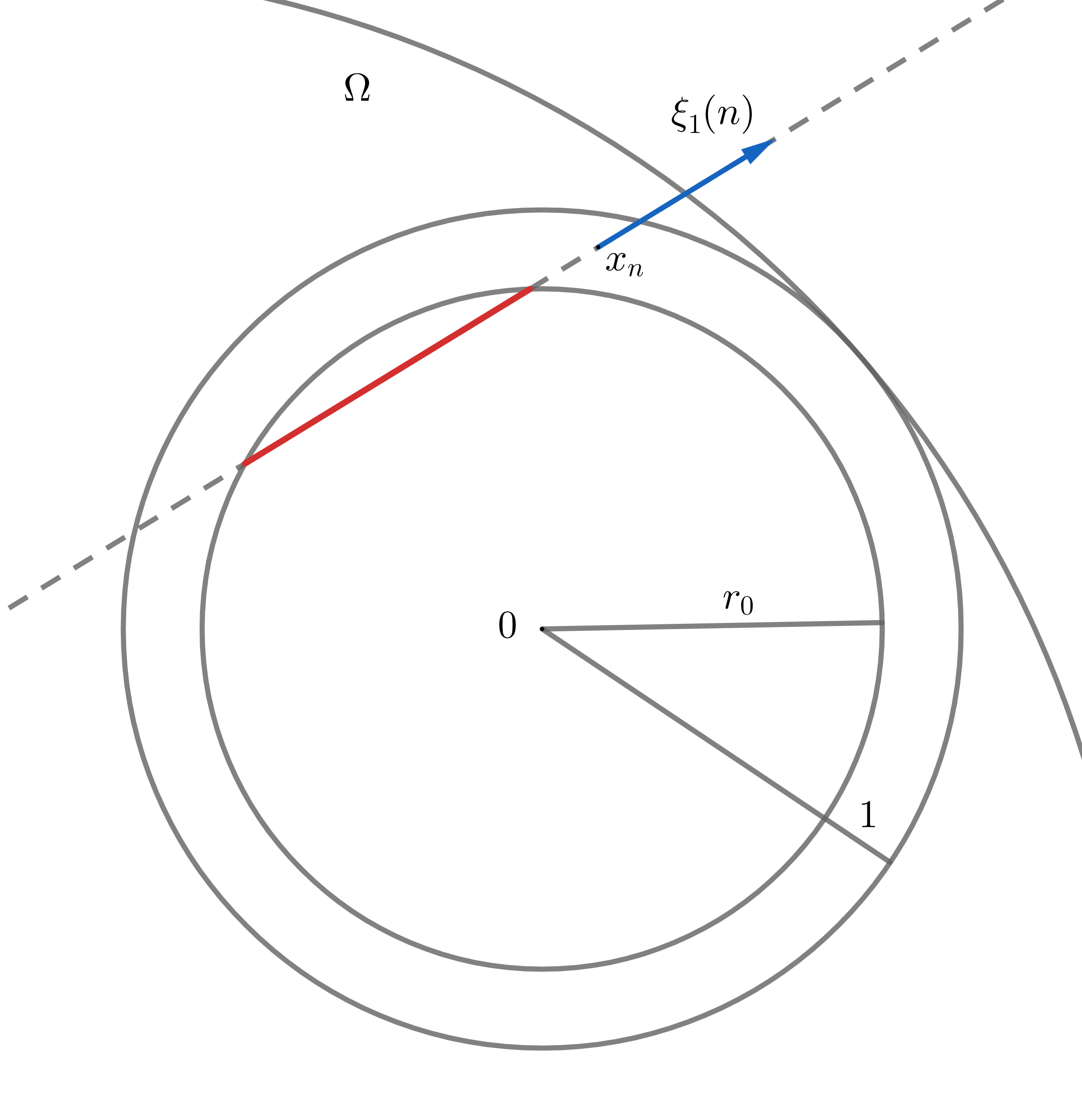}
\caption{The blue vector represents $\xi_1(n)$, and the red segment corresponds to points $x_n -\tau \xi_1(n)$, with $\tau \in [\tau_1(n), \tau_2(n)]$. }
\label{tau}
\end{figure}

Notice that if $\tau \in [\tau_1(n), \tau_2(n)]$ then $x_n - \tau \xi_1(n) \in B_{r_0}(0)$, as 
\[ \abs{x_n -\tau \xi_1(n)}^2 \le \abs{x_n}^2+ \tau^2 - \frac{2 \tau \abs{x_n}}{\sqrt N} \le 1 -\frac{1}{2N}, \] 
see also Figure \ref{tau}. 
Also, for $n$ large we can assume $\rho=d(x_n) < \tau_1(n) < \tau_2(n)$, since as $n\to+\infty$, $d(x_n)\to0$, $\tau_1(n)\to\frac{1}{\sqrt{N}}\left(1-\frac{1}{\sqrt{2}}\right)$ and $\tau_2(n)\to\frac{1}{\sqrt{N}}\left(1+\frac{1}{\sqrt{2}}\right)$.

Integrals $J_1$ and $J_3$ can be estimated once again as above, exploiting the inequality
\[ \delta(u, x_n, \tau \xi_1(n)) \ge \delta(w_n, x_n, \tau \xi_1(n)).\]
In order to estimate $J_2$, we now use the fact that 
$u(x_n - \tau \xi_1 (n))\geq\min_{\overline B_{r_0}} u >0$. 
We obtain
\begin{align*} J_2 & \ge  \int_{\tau_1(n)}^{\tau_2(n)}  \frac{u(x_n - \tau \xi_1(n)) - 2u(x_n)}{\tau^{1+2s}}\,d\tau \\
 &\ge \left( \min_{\overline B_{r_0}} u-2u(x_n)\right)\,  \int_{\tau_1(n)}^{\tau_2(n)}  \frac{1}{\tau^{1+2s}}= \frac {\min_{\overline B_{r_0}} u-2u(x_n)}{2s} \, \left( \frac{1}{\tau_1(n)^{2s}} - \frac{1}{\tau_2(n)^{2s}} \right). 
\end{align*}

Now, putting estimates above together and recalling \eqref{u supersol hopf}, one has
\begin{equation}\label{eq:hopf} \begin{split} \frac 1n \ge & \sum_{i=1}^N \int_0^{+\infty}  \frac{\delta(w_n, x_n, \tau \xi_i(n))}{\tau^{1+2s}}\,d\tau  - \int_{\tau_1(n)}^{\tau_2(n)}  \frac{\delta(w_n, x_n, \tau \xi_1(n))}{\tau^{1+2s}}\,d\tau  \\
&\hspace{3cm}+ \frac {\min_{\overline B_{r_0}} u-2u(x_n)}{2s} \, \left( \frac{1}{\tau_1(n)^{2s}} - \frac{1}{\tau_2(n)^{2s}} \right)\,.\end{split}\end{equation}

Notice that, as $n\to+\infty$
\[ \abs{\int_{\tau_1(n)}^{\tau_2(n)}  \frac{\delta(w_n, x_n, \tau \xi_1(n))}{\tau^{1+2s}} \, d \tau} 
\le \frac{2}{s \, n} \left( \frac{1}{\tau_1(n)^{2s}} - \frac{1}{\tau_2(n)^{2s}} \right) \to 0, \]
and that by Lemma \ref{funzione barriera}
\[ \sum_{i=1}^N \int_0^{+\infty}  \frac{\delta(w_n, x_n, \tau \xi_i(n))}{\tau^{1+2s}} \, d \tau = - \frac{N}{n} C_s \beta(1-s, s). \]
Thus by taking the limit $n \to +\infty$ in \eqref{eq:hopf} and using \eqref{2406eq3} we get the contradiction
\[ 0 < \frac 1{2s} \min_{\overline B_{r_0}} u \, \left( \left( \frac{1}{\sqrt N} \left( 1 - \frac{1}{\sqrt 2} \right) \right)^{-2s} -  \left( \frac{1}{\sqrt N} \left( 1 +\frac{1}{\sqrt 2} \right) \right)^{-2s} \right) \le 0 \,. \qedhere\] 
\end{proof}

\section{Stability and the Perron method}\label{sec:Perron}
We now give some stability results which will be crucial for our purposes. They have been treated in a very general context in \cite{BCI, BarlesImbert}, see also \cite{AlvarezTourin}, here we give a simplified proof with full details for the operators $\I_k^\pm$. 

For the local counterparts, we refer to \cite{CIL}. 
Let us set 
\[ u_* (x)= \sup_{r>0} \inf_{\abs{y-x} \le r} u(y), \quad u^*(x) = \inf_{r>0} \sup_{\abs{y-x} \le r} u(y) \]
and 
\[ {\liminf}_* u_n(x)=\lim_{j \to \infty} \inf \left\{u_n(y): n \ge j, \, \abs{y-x} \le \frac1j \right\} , \]
\[ {\limsup}^* u_n (x)= \lim_{j \to \infty} \sup \left\{u_n(y): n \ge j, \, \abs{y-x} \le \frac1j \right\} .\]

\begin{lemma}\label{stability1}
Let $u_n \in USC(\Omega)$ (respectively $LSC(\Omega)$) be a sequence of subsolutions (supersolutions) of
\begin{equation}\label{limit j} \I_k^\pm u_n = f_n(x)  \text{ in } \Omega,  \end{equation}
where $f_n$ are locally uniformly bounded functions, 
and $u_n \le 0$ ($u_n \ge 0$) in $\R^N \setminus \Omega$. We assume that there exists $M>0$ such that  for any $n \in \N$ 
\begin{equation}\label{bound uj} \norm{u_n}_\infty \le M \text{ in } \R^N. \end{equation}
Then $\overline u:= {\limsup}^* u_n$ (resp. $\underline u :={\liminf}_* u_n$) is a subsolution (resp. supersolution) of 
\[ \I_k^\pm u = f(x) \text{ in } \Omega,  \]
such that $u \le 0$ ($u \ge 0$)  in $\R^N \setminus \overline \Omega$, where $\underline f=\liminf_* f_n$ (resp. $\overline f=\limsup^* f_n$).
\end{lemma}
\begin{remark}
Notice that in general we cannot guarantee that the limit solution $\overline u$ is $\le 0$ also on the boundary of the domain $\Omega$. However, in our next results, we will always be able to avoid this difficulty, by comparing the limit solution with the distance function to the boundary, see also Lemma \ref{barrier} below. 
\end{remark}
\begin{proof}
Let us only consider $\I_k^+$, for $\I_k^-$ is analogous. 
Let us fix $x_0 \in \Omega$, and let us choose $\Phi \in C^2(B_\rho(x_0))$ such that $\Phi(x_0)=\overline u (x_0)$, and $\Phi > \overline u$ in $B_\rho(x_0) \setminus \{ x_0 \}$. We can choose $x_n \to x_0$ such that up to a subsequence $u_n-\Phi$ has a maximum in $x_n$ in $\overline B_{\rho/2}(x_n)$, and $\overline u(x_0)=\lim_n u_n(x_n)$. 
Since $u_n$ are subsolutions, there exist $\{ \xi_i(n) \} \in \mathcal{V}_k$ such that 
\begin{equation}\label{eq:stability} f_n(x_n) - \frac 1n \le \sum_{i=1}^k \left( \int_{0}^{\rho/2} \frac{\delta(\Phi, x_n, \tau \xi_i(n))}{\tau^{1+2s}} \, d\tau + \int_{\rho/2}^{+\infty} \frac{\delta(u_n, x_n, \tau \xi_i(n))}{\tau^{1+2s}}\, d\tau  \right) \end{equation}
Up to extracting a further subsequence, we can assume $\xi_i(n) \to \bar \xi_i$ as $n \to \infty$. Then, recalling $\Phi \in C^2(B_\rho(x_0))$, 
\[ 
\lim_{n\to +\infty} \int_{0}^{\rho/2} \frac{\delta(\Phi, x_n, \tau \xi_i(n))}{\tau^{1+2s}}\, d\tau = \int_{0}^{\rho/2} \frac{\delta(\Phi, x_0, \tau \bar \xi_i)}{\tau^{1+2s}} \, d\tau.  \] 
On the other hand, by applying Fatou Lemma, and using hypothesis \eqref{bound uj}, 
\[ \limsup_{n\to +\infty} \int_{\rho/2}^{+\infty} \frac{\delta(u_n, x_n, \tau \xi_i(n))}{\tau^{1+2s}} \, d\tau \le \int_{\rho/2}^{+\infty} \frac{\delta(\overline u, x_0, \tau \bar \xi_i)}{\tau^{1+2s}} \, d\tau \]
Thus, recalling \eqref{eq:stability}, passing to the limit, and also using that $\Phi \ge \overline u$ in $B_\rho(x_0)$, 
\[ \begin{split} \underline f(x_0) &\le \sum_{i=1}^k\left(  \int_{0}^{\rho/2} \frac{\delta(\Phi, x_0, \tau \bar \xi_i)}{\tau^{1+2s}}\, d\tau+ \int_{\rho/2}^{+\infty} \frac{\delta(\overline u, x_0, \tau \bar \xi_i)}{\tau^{1+2s}}\, d\tau \right) \\
&\le \sum_{i=1}^k\left(  \int_{0}^{\rho} \frac{\delta(\Phi, x_0, \tau \bar \xi_i)}{\tau^{1+2s}}\, d\tau+ \int_{\rho}^{+\infty} \frac{\delta(\overline u, x_0, \tau \bar \xi_i)}{\tau^{1+2s}} \, d\tau\right) \end{split} \]
which implies the conclusion. 
\end{proof}

Analogously one proves 
\begin{lemma}\label{stability2}
Let $(u_\alpha)_\alpha \subseteq USC(\Omega)$ (respectively $LSC(\Omega)$) a family of subsolutions (supersolutions) of 
\[ \I_k^\pm u_\alpha = f_\alpha(x)\text{ in } \Omega  \]
such that 
$u_\alpha \le 0$ ($u_\alpha \ge 0$)  in $\R^N \setminus \Omega$, and  there exists $M>0$ such that for any $\alpha $
\[ \norm{u_\alpha}_\infty \le M \text{ in } \R^N,  \]
where $f_\alpha$ are uniformly bounded. 
Set $u=\sup_\alpha u_\alpha$ (resp. $v=\inf_{\alpha} u_\alpha$). Then $u^*$ (resp $v_*$) is a subsolution (resp supersolution) of 
\[ \I_k^\pm u = f(x) \text{ in } \Omega  \]
such that $u \le 0$ ($u \ge 0$)  in $\R^N \setminus \Omega$, 
where $f=(\inf_\alpha f_\alpha)_*$ (resp. $f=(\sup_\alpha f_\alpha)^*$).
\end{lemma}

As a consequence, we get the following analog of the Perron method.
\begin{lemma}\label{perron}
Let $\underline u$ and $\overline u$ in $C(\R^N)$ be respectively sub and supersolutions of 
\begin{equation}\label{eq perron}
\I_k^\pm u =f(x) \text{ in } \Omega,
\end{equation}
such that $\underline u= \overline u=0$ in $\R^N \setminus \Omega$. Then there exists a solution $v \in C(\R^N)$ to \eqref{eq perron} such that $\underline u  \le v \le \overline u$, and $v=0$ in $\R^N \setminus \Omega$. 
\end{lemma}
\begin{proof}
In what follows we only consider the case $\I_k^+$, similar considerations hold for $\I_k^-$.
Let 
\[ v= \sup \{ u: \, u \text{ is a subsolution to \eqref{eq perron} s.t. } u \le \overline u \text{ in } \R^N\}. \]
Notice that $v \in L^\infty(\R^N)$ as
\[ \underline u \le v_* \le v \le v^* \le \overline u, \]
which also implies $v=0$ in $\R^N \setminus \Omega$. 
We know by Lemma \ref{stability2} that $v^*$ is a subsolution to \eqref{eq perron}, thus $v^* \le v$ by maximality of $v$ and $v=v^*$. We claim that $v_*$ is a supersolution to \eqref{eq perron}. If the claim is true, then by the comparison principle Theorem \ref{comparison} we conclude $v^* \le v_*$, and since the other inequality trivially holds, then $v=v_*=v^* \in C(\R^N)$ is a solution to \eqref{eq perron} such that $v=0$ in $\R^N \setminus \Omega$. 

We now prove the claim. Let us assume by contradiction that $v_*$ is not a supersolution. Then, there exists $x_0 \in \Omega$, $\rho >0$ and $\Phi \in C^2(\overline{B_\rho(x_0)})$ such that $\Phi(x_0)=v_*(x_0)$, $\Phi < v_*$ in $\overline{B_\rho(x_0)} \setminus \{ x_0 \}$, and
\begin{equation}\label{hp assurdo} 
\I_k^+ \Psi(x_0) > f(x_0), \end{equation}
where $\Psi \in LSC(\R^N) \cap L^\infty(\R^N) \cap C^2(B_\rho(x_0))$ is defined as 
\[ \Psi(x)=
\begin{cases}
\Phi(x) & \text{ if } x \in \overline B_\rho(x_0) \\
v_*(x) & \text{ if } x \in \R^N \setminus  \overline B_\rho(x_0). 
\end{cases} \]

By Proposition \ref{semicont}, there exist $r < \rho/2$ and $\varepsilon_0 >0$ such that
\begin{equation}\label{perron dis2} \I_k^+ \Psi(x) \ge  f(x)+\varepsilon_0 \end{equation}
for any $x \in B_r(x_0)$. 
Moreover, for any $\eta >0$ let 
\[ \Psi_{\eta}(x)=
\begin{cases}
\Phi(x) +\eta & \text{ if } x \in \overline B_\rho(x_0) \\
v_*(x) & \text{ if } x \in \R^N \setminus  \overline B_\rho(x_0). 
\end{cases} \]
Then, 
\begin{equation}\label{perron dis} \I_k^+ \Psi_{\eta}(x) \ge  f(x) \end{equation}
for any $\eta < \eta_1=\varepsilon_0 C_s^{-1}  \frac s k \left( \frac \rho 2 \right)^{2s}$ and for any $x \in B_r(x_0)$. 
Indeed, notice that $\Psi_{\eta} = \Psi + \eta \chi_{\overline B_\rho(x_0)}$, where $\chi_A$ is the characteristic function of the set $A$,
and that for any $\abs{\xi}=1$ and $x \in B_r (x_0)$, $x \pm \tau \xi \in B_\rho(x_0)$ if $\tau < \rho-r$. Thus, 
by direct computations 
\begin{align*} \I_\xi \chi_{\overline B_\rho(x_0)} (x) & = C_s \int_{\rho-r}^{+\infty} \frac{\delta(\chi_{\overline B_\rho(x_0)} , x, \tau \xi)}{\tau^{1+2s}} \, d \tau \ge -2 C_s \int_{\rho-r}^{+\infty}  \frac 1{\tau^{1+2s}}  \, d \tau \\
&=  -\frac{C_s }s (\rho- r)^{-2s} 
\ge -\frac{C_s }s \left( \frac \rho 2 \right)^{-2s}. \end{align*}
Thus
\[ \I_k^+ \Psi_{\eta}(x) \ge \I_k^+ \Psi(x) - C_s \frac ks \left( \frac \rho 2 \right)^{-2s} \eta \ge f(x) + \varepsilon_0-C_s \frac ks \left( \frac \rho 2 \right)^{-2s}\eta \ge f(x) \] 
by using \eqref{perron dis2}. 

Let us take 
\[ \eta_2= \min_{\overline B_{\rho}(x_0) \setminus B_{r/2}(x_0)} (v_* - \Phi)>0, \]
so that $v_* > \Phi+\eta$ in $\overline B_{\rho}(x_0) \setminus B_{r/2}(x_0)$ for any $\eta < \eta_2$. 

Consider 
\[ \eta_0\le\min \{\eta_1, \eta_2 \}. \]

Define
\[ w= \begin{cases}
\max \{ v, \Psi_{\eta_0} \} &\text{ in } B_{r}(x_0)\\
v &\text{ in } \R^N \setminus B_{r}(x_0). 
\end{cases} \]
In particular,  $w(x) \ge \Psi_{\eta_0}(x) $ for all $x$. 

Let us prove that $w$ is a subsolution. 
Let us fix $\bar x \in \Omega$, and let us choose $\varphi \in C^2(B_\varepsilon(\bar x))$ such that $w(\bar x)=\varphi(\bar x)$, and $w(x) \le \varphi(x)$ in $B_\varepsilon(\bar x)$. 

If $w(\bar x )=v(\bar x)$, then $\varphi$ is a test function for $v$, and we exploit the fact that $v$ is a subsolution.
If $w(\bar x)=\Phi(\bar x)+ \eta_0> v(\bar x)$, then  in particular $\bar x \in B_{r/2}(x_0)$. Set
\[ \theta(x)=
\begin{cases}
\varphi(x) &\text{ if } x \in B_\varepsilon(\bar x) \\
w(x) &\text{ if } x \in \R^N \setminus B_\varepsilon(\bar x).
\end{cases} \]
One has
\[ \theta(\bar x)=\varphi(\bar x)=w(\bar x)=\Phi(\bar x)+\eta_0=\Psi_{\eta_0}(\bar x).  \]
Also, $\theta(x)\ge\Psi_{\eta_0}(x)$ for any $x$. Indeed, if $x \in B_\varepsilon(\bar x)$, then $\theta(x) =\varphi(x) \ge w(x) \ge \Psi_{\eta_0}(x)$, whereas if $x \not \in B_\varepsilon(\bar x)$, then $\theta(x)=w(x) \ge  \Psi_{\eta_0}(x)$. Therefore,
\[ \I_k^+ \theta(\bar x) \ge \I_k^+ \Psi_{\eta_0} (\bar x) \ge f(\bar x) \]
by \eqref{perron dis}. 

Hence $w$ is a subsolution, and this yields a contradiction. Indeed, 
there exists a sequence $x_n \to x_0$ such that $\lim_{n\to \infty} v(x_n)=v_*(x_0)$, and one has
\[ \lim_n (w(x_n)-v(x_n) )= \max \{ v_*(x_0), \Phi(x_0)+\eta_0 \} - v_*(x_0) = \eta_0 >0. \]
Thus, $w(x) > v(x)$ for some $x$. Finally, we notice that $w \le \overline u$ by comparison, and as a consequence $w \le v$ by maximality of $v$, a contradiction. 
\end{proof}

We finally prove existence of a unique solution to the Dirichlet problem in uniformly convex domains
$$
\Omega=\bigcap_{y\in Y}B_R(y).
$$
 The proof will be based on stability properties above.

\begin{theorem}\label{lem dirichlet}
Let $f$ be a bounded continuous function, and let $\Omega$ be a uniformly convex domain. Then there exists a unique function $u \in C(\R^N)$ such that 
\begin{equation}\label{dirichlet} \begin{cases}
\I_k^\pm u = f(x) &\text{ in } \Omega \\
u=0 &\text{ in } \R^N \setminus \Omega.
\end{cases} \end{equation}
\end{theorem}
\begin{proof}
Exploiting the barrier functions in Lemma \ref{funzione barriera}, we build suitable sub/super solutions. Indeed, for any $y \in Y$ one considers the function 
\[ v_y(x)=M(R^2-\abs{x-y}^2)^s_+ \]
which for $M=M(k, s)$ big enough satisfies
\[ \I_k^+ v_y \le - \norm{f}_\infty \text{ in } B_R(y). \]
We now take 
\begin{equation}\label{v inf} v(x)= \inf_{y \in Y} v_y(x) \end{equation}
which is a supersolution to \eqref{dirichlet}. In order to prove it, first we note that $0 \le v(x) \le M R^{2s}$, hence $v$ is bounded. Moreover, 
notice that $v \in C^{0, s}(\R^N)$. Indeed, for any $x, y \in \overline \Omega$, one has 
\[ \begin{split} \abs{v(x)-v(z)} &\le \sup_y \abs{v_y(x) - v_y(z)} \\
&=M \sup_y \abs{(R^2 - \abs{x-y}^2)^s-(R^2 - \abs{z-y}^2)^s }\\
& \le M \sup_y \abs{(R^2 -\abs{x-y}^2) - (R^2 -\abs{z-y}^2)}^s \\
&= M \sup_y \abs{\abs{z-y}^2 - \abs{x-y}^2}^s \\
& = M \sup_y (\abs{z-y} + \abs{x-y})^s \abs{\abs{z-y}-\abs{x-y}}^s \\
& \le M (2R)^s \abs{z-x}^s.
\end{split}\]
Moreover, $v = 0$ in $\R^N \setminus \Omega$. Indeed, if $x \not \in \Omega$, there exists $y=y(x)$ such that $x \not \in B_R(y)$ which implies 
\[ 0  \le v(x) \le v_y(x) =M(R^2-\abs{x-y}^2)_+^s=0. \]

The infimum in definition \eqref{v inf} is attained, as given $x_0 \in \Omega$, we can choose $y_0 \in Y$ and $z_0 \in \partial B_R(y_0)$ such that 
\[ \abs{x_0 - z_0} = d(x_0)= \eta. \]
Therefore, as $B_\eta(x_0) \subseteq \Omega \subseteq B_R(y)$ for any $y \in Y$, 
\[ \abs{y-x_0} \le R-\eta = \abs{x_0-y_0} \]
and as a consequence $v(x_0)=v_{y_0}(x_0)$. In particular,
\[  \I_k^+ v_{y_0}(x_0) \le - \norm{f}_\infty, \]
which yields
\[ \I_k^+ v(x) \le- \norm{f}_\infty \, \text{ in } \Omega. \]

Analogously we take the supremum of the sub solutions 
\[ w_y(x)=- v_y(x). \]
Notice that 
\[\I_k^+ w_y(x) \ge \I_k^- w_y(x)=-\I_k^+ v_y(x) \ge \norm{f}_\infty \text{ in } B_R(y) \]
for a sufficiently big constant $M$. 

We now exploit the Perron method, applying Lemma \ref{perron}, to get a solution to \eqref{dirichlet}. 
Uniqueness follows from Theorem \ref{comparison}. 
\end{proof}

\section{Maximum principles and principal eigenvalues}\label{sec:maximum}
We finally define the following generalized principal eigenvalues, adapting the classical definition in \cite{BNV}, 
\[ \mu_k^\pm = \sup \left \{ \mu :\, \exists v \in LSC(\Omega)\cap L^\infty(\R^N), v>0 \text{ in } \Omega, v \ge 0 \text{ in } \R^N, \I_k^\pm v + \mu v \le 0 \text{ in } \Omega \right \}. \]
Also let us set 
\[ \bar\mu^\pm_k=\sup\left\{\mu:\,\exists v\in LSC(\Omega)\cap L^\infty(\R^N),\,\inf_\Omega v>0,\,v\geq0\;\text{in $\R^N$},\;\I^\pm_kv+\mu v\leq0  \text{ in } \Omega\right\}. \]

\begin{remark}
In this section, we only consider the operators $\I^\pm_k(\cdot)+\mu\cdot$, however, one can also treat operators with a zero order term like $\I^\pm_k(\cdot)+c(x) \cdot + \mu\cdot$, up to some technicalities. 
\end{remark}
\begin{theorem}\label{max principle bar}
The operators $\I^\pm_k(\cdot)+\mu\cdot$ satisfy the maximum principle for $\mu<\bar\mu^\pm_k$.
\end{theorem}
\begin{proof}
We consider $\I^+_k$, the other case being analogous. Let $\mu<\bar\mu^+_k$ and let $u\in USC(\overline\Omega)\cap L^\infty(\R^N)$ be a solution of 
\begin{equation*}
\left\{\begin{array}{cl}
\I^+_ku+\mu u\geq0 & \text{in $\Omega$}\\
u\leq0 & \text{in $\R^N\backslash\Omega$}.
\end{array}\right.
\end{equation*}
By contradiction we suppose that $u(x_0)>0$ for some $x_0\in\Omega$. In view of Theorem \ref{comparison} we have $\mu>0$.
 By the definition of $\bar\mu^+_k$ there exists $\eta\in(\mu,\bar\mu^+_k)$ and a nonnegative  bounded function $v\in LSC(\Omega)$ such that 
\begin{equation*}
\I^+_kv+\eta v\leq0\quad\text{in $\Omega$\; and\; $\displaystyle \inf_\Omega v>0$.}
\end{equation*}
Set $\gamma=\sup_\Omega\frac uv$. Then $$0<\frac{u(x_0)}{v(x_0)}\leq\gamma<+\infty$$
and for any $\varepsilon\in(0,\gamma)$ there exists $z_\varepsilon\in\Omega$ such that 
$$
u(z_\varepsilon)-(\gamma-\varepsilon)v(z_\varepsilon)>0.
$$
From this we infer that there exists $x_\varepsilon\in\Omega$ such that 
$$
M_\varepsilon:=\max_{\overline\Omega} [u(x)-(\gamma-\varepsilon)v(x)] =u(x_\varepsilon)-(\gamma-\varepsilon)v(x_\varepsilon)>0.
$$
For $n\in\mathbb N$ let $x_n=x_n(\varepsilon), y_n=y_n(\varepsilon)\in \overline\Omega$ be such that 
\begin{equation}\label{2605eq2}
\begin{split}
\max_{\overline\Omega\times\overline\Omega}[u(x)-(\gamma-\varepsilon)v(y)-n|x-y|^2]&=u(x_n)-(\gamma-\varepsilon)v(y_n)-n|x_n-y_n|^2\\
&\geq M_\varepsilon>0.
\end{split}
\end{equation}
Arguing as in the proof of Theorem \ref{comparison} we find that, for $n$ sufficiently large, 
\begin{equation}\label{2605eq1}
\max_{\overline\Omega\times\overline\Omega}[u(x)-(\gamma-\varepsilon)v(y)-n|x-y|^2]=\max_{\R^N\times\R^N}[u(x)-(\gamma-\varepsilon)v(y)-n|x-y|^2].
\end{equation}
Moreover, up to extract a subsequence, we may further assume that $(x_n,y_n)\to(\bar x,\bar x)$, with $\bar x\in\Omega$. Using $\varphi_n(x)=u(x_n) + n|x-y_n|^2- n\abs{x_n-y_n}^2$ as test function for $u$ at $x_n$, and also testing $v$ at $y_n$ with $\phi_n(y)=(\gamma-\varepsilon)v(y_n) -n|x_n-y|^2 + n\abs{x_n-y_n}^2$, and finally subtracting the corresponding inequalities, see also the proof of Theorem \ref{comparison}, we obtain
\begin{equation*}
\begin{split} 
\eta(\gamma-\varepsilon)v(y_n)&\leq\mu u(x_n)+(\gamma-\varepsilon+1)\frac{nk\rho^{2-2s}}{1-s}\\
&\hspace{1.7cm}+C_s \sup_{\left\{\xi_i\right\}_{i=1}^k\in{\mathcal V}_k}\sum_{i=1}^k\int_\rho^{+\infty}\frac{\delta(u, x_n, \tau \xi_i)- \delta((\gamma-\varepsilon) v, y_n, \tau \xi_i)}{\tau^{1+2s}}\,d\tau.
\end{split}
\end{equation*}
By \eqref{2605eq2}-\eqref{2605eq1} it follows that $\delta(u, x_n, \tau \xi_i)- \delta((\gamma-\varepsilon) v, y_n, \tau \xi_i)\leq0$. Hence
\begin{equation*}
\eta(\gamma-\varepsilon)v(y_n)\leq\mu u(x_n)+(\gamma-\varepsilon+1)\frac{nk\rho^{2-2s}}{1-s}.
\end{equation*}
Letting $\rho\to0$
\begin{equation*}
\eta(\gamma-\varepsilon)v(y_n)\leq\mu u(x_n).
\end{equation*}
Then as $n\to+\infty$
\begin{equation*}
\eta(\gamma-\varepsilon)v(\bar x)\leq\liminf_{n\to+\infty}\eta(\gamma-\varepsilon)v(y_n)\leq\limsup_{n\to+\infty}\mu u(x_n)\leq\mu u(\bar x)\leq\mu\gamma v(\bar x).
\end{equation*}
Since $v$ and $\gamma$ are positive and $\varepsilon$ can be chosen arbitrarily small, we reach the contradiction
\[
\eta\leq\mu. \qedhere
\]
\end{proof}

\begin{proposition}\label{mu-inf}
One has
\begin{enumerate}
\item[(i)] $\bar\mu_k^-=\mu_k^-=+\infty$ for any $k < N$. 
\item[(ii)] If $B_{R_1} \subseteq \Omega \subseteq B_{R_2}$, then 
\[ 0<\frac{c_2}{R_2^{2s}} \le \bar\mu_1^+ \le \dots \le \bar\mu_N^+\le \bar\mu_N^- \le \frac{c_1}{R_1^{2s}} < +\infty,  \]
where $c_1, c_2$ are positive constants depending on $s$.
\end{enumerate}
\end{proposition}
\begin{proof}
(i) Let $w(x)=e^{-\alpha \abs{x}^2} > 0$ for $\alpha>0$ and fix any $\mu >0$. Since 
\[ \int_0^{+\infty} (1-e^{-\alpha \tau^2}) \tau^{-(1+2s)} \, d\tau = \alpha^s \int_0^{+\infty} (1-e^{- \tau^2}) \tau^{-(1+2s)} \, d\tau, \]
using Theorem 3.4 in \cite{BGT}  (see also Remark 3.5) we obtain 
\begin{align*} \I_k^- w + \mu w &=  k \mathcal{I}_{x^\perp} w + \mu w \\
&= - 2k C_s e^{-\alpha \abs{x}^2} \int_0^{+\infty} (1-e^{-\alpha \tau^2}) \tau^{-(1+2s)} \, d\tau + \mu e^{-\alpha \abs{x}^2} = 0
\end{align*}
if $$\alpha^s=\frac{\mu}{2k C_s \int_0^{+\infty} (1-e^{-\tau^2}) \tau^{-(1+2s)}},$$ 
where $x^\perp$ is a unitary vector such that $ \langle x, x^\perp \rangle=0$.

(ii) We first note that in the definitions of $\bar\mu^\pm_k$ it is not restrictive to suppose $\mu\geq0$ (since the constant function $v=1$ is a positive solution of $\I^\pm_kv=0$). Moreover if $\mu\geq0$ and $v$ is a nonnegative supersolution of the equation $$\I^+_kv+\mu v=0 \quad\text{in $\Omega$},$$
then $\I^+_kv\leq0$ in $\Omega$ and using Remark \ref{rmk:ordine} we have 
$$
\I^+_{k+1}v+\mu v\leq0 \quad\text{in $\Omega$}.
$$
This leads to $\bar\mu^+_k\leq\bar\mu^+_{k+1}$ for any $k=1,\dots,N-1$. If $k=N$, using the inequality $\I^-_N\leq\I^+_N$ we immediately obtain that $\bar\mu^+_N\leq\bar\mu^-_N$.
 
Also, by scaling we obtain
\[ \bar\mu_N^- (\Omega) \le \bar\mu_N^- (B_{R_1}) =  \frac{\bar\mu_N^- (B_1)}{R_1^{2s}}. \]
Hence it is sufficient to prove that $\bar\mu_N^- (B_1)$ is bounded from above. 

Arguing as in \cite{QSX}, choose a constant function $h \ge 0$, $h \not \equiv 0$ with compact support in $B_1$. By Theorem \ref{lem dirichlet}, there exists a unique solution to the following
\[ \begin{cases}
- \I_N^- v = h &\text{ in }B_1 \\
v=0 &\text{ in } \R^N \setminus B_1. 
\end{cases} \]
By Theorem \ref{comparison} and Theorem  \ref{SMP} $v >0$ in $B_1$. Since $h$ has compact support we may select a constant $\rho_0 >0$ such that $\rho_0 v \ge h$ in $B_1$. Therefore, $v$ satisfies
\[ \begin{cases}
 \I_N^- v +\rho_0v\geq0 &\text{ in }B_1 \\
v=0 &\text{ in } \R^N \setminus B_1. 
\end{cases} \]
By Theorem \ref{max principle bar} we infer that $\bar \mu_N^- \le \rho_0$. 

As for the bound from below, we observe that  $u(x)={(R_2^2 - \abs{x}^2)}^s_+ + \varepsilon$ satisfies 
\[ \I_1^+ u + \mu u=- C_s \beta(1-s, s) + \mu u \le 0 \]
if we take $\mu \le \frac{C_s \beta(1-s,s)}{R_2^{2s}+\varepsilon}$ for any $\varepsilon>0$, thus $\bar \mu_1^+ \ge \frac{ C_s \beta(1-s,s)}{R_2^{2s}}>0$. 
\end{proof}

\begin{remark}
Notice that the proof of $(i)$ above suggests the existence of a continuum of eigenvalues in $(0, +\infty)$ for $\I_k^- + \mu$ in $\mathbb R^N$.
\end{remark}

We now consider uniformly convex domains
 and  prove that $\bar \mu_k^+=\mu_k^+$. Moreover this common value turns out to be the optimal threshold for the validity of the maximum principle. 
We start with the next Lemma which will be crucial in the rest of the paper. 

\begin{lemma}\label{barrier}
Let $m$ be a positive constant and let $u$ be a solution of
\[ \begin{cases}
\I_k^+u (x) \ge -m &\text{ in } \Omega\\
u \le 0 &\text{ in } \R^N\setminus \Omega, 
\end{cases} \] 
where the domain $\Omega$ is  uniformly convex. Then there exists a positive constant $C=C(\Omega, m, s)$ such that 
\begin{equation}\label{2406eq12} u(x) \le C \, d(x)^s \end{equation}
for any $x \in \overline \Omega$.
\end{lemma}
\begin{proof}
Fix any $y \in Y$ and consider the function 
\[ v_y(x)=M {(R^2-\abs{x-y}^2)}^s _+\]
 where $M$ is such that $k M C_s \beta(1-s, s)=m$. Then 
\[ \I_k^+ v_y(x) = - k  M C_s \beta(1-s, s)=-m. \]
Also, we point out that $v_y(x) \ge 0$ in $\R^N$. By the comparison principle, see Theorem \ref{comparison}, 
$u(x) \le v_y(x)$ in $\R^N$. 
Let $x \in \Omega$ and select $z \in \partial \Omega$ so that $d(x)=\abs{x-z}$. Choose $y \in Y$ such that $z \not \in B_R(y)$.
Notice that, since $\abs{x-y} \le R$, 
\begin{align*} {(R^2-\abs{x-y}^2)}^s &= {(R-\abs{x-y})}^s{(R+\abs{x-y})}^s \le 2^s R^s {(R-\abs{x-y})}^s \\
&=  2^s R^s \abs{x-z}^s= 2^s R^s d(x)^s.\end{align*}
Thus for any $x \in \overline{\Omega}$ 
\[ u(x) \le M {(R^2-\abs{x-y}^2)}^s \le M 2^s R^s d(x)^s, \]
leading to \eqref{2406eq12} with $C= M 2^s R^s$.
\end{proof}

\begin{theorem}\label{subsol mu} 
Let $\Omega$ be a uniformly convex domain. There exists a nonnegative subsolution $v \not\equiv 0$ of 
\[ \begin{cases}
\I_k^+ v + \bar \mu_k^+ v= 0 &\text{ in } \Omega \\
v =0 &\text{ in } \R^N \setminus \Omega.
\end{cases} \]
\end{theorem}
\begin{proof}

Let us consider the problem 
\begin{equation}\label{equation n}
 \begin{cases}
\I_k^+ w +\left(\bar \mu_k^+ -  \frac1n \right) w=-1 &\text{ in } \Omega \\
w =0 &\text{ in } \R^N \setminus \Omega,
\end{cases} \end{equation}
and define
\[ A_n=\{ w \in USC(\R^N)  \text{ nonnegative subsolution of \eqref{equation n} s.t. } w=0 \text{ on } \R^N \setminus \Omega \}. \]
One has $\emptyset \ne A_n \subseteq A_{n+1}$. We claim that for any $n$ there exists $w_n \in A_n$ such that $\lim_n \norm{w_n}_\infty =+ \infty$. 
If the claim is true, then we define $z_n=\frac{w_n}{\norm{w_n}}$, which turn out to be solutions of
\[ \I_k^+ z_n +\left(\bar \mu_k^+ - \frac1n \right) z_n \ge - \frac{1}{\norm{w_n}} \text{ in } \Omega. \]
By semicontinuity, there exists a sequence $x_n \in \Omega$ such that $\sup_\Omega z_n=z_n(x_n)=1$. 
Up to a subsequence, $x_n \to x_0$, and by Lemma \ref{barrier} $x_0 \in \Omega$. Thus
$v(x)={\limsup_n}^* z_n(x) $
satisfies by Lemma \ref{stability1}
\[ \I_k^+ v + \bar \mu_k^+ v \ge  0 \text{ in } \Omega \]
and, again by Lemma \ref{barrier} $v = 0$ on $\R^N \setminus \Omega$. Also, $v(x_0)=1$ and the proof is complete.

Let us now prove the claim. We will proceed by contradiction, assuming that for any sequence $u_n \in A_n$ then $\limsup_n \norm{u_n}_\infty  < + \infty$, and split the proof into steps.

\textit{Step 1.} 
We show that $U_n(x)=\sup_{w \in A_n} w(x) < +\infty$ for any $x$ and any $n$. 

\noindent If it is not the case, then there exists $\bar n$ and $\bar x$ such that $U_{\bar n}(\bar x)=+\infty$, and by definition of supremum, there exists a sequence $(u_n)_n \subseteq A_{\bar n}$ such that $\lim_n u_n (\bar x) =+ \infty$. Since for any $n \ge \bar n$ one has $A_{\bar n} \subseteq A_n$, then $u_n \in A_n$ for any $n \ge \bar n$ and $\lim_n \norm{u_n}_\infty=+\infty$, a contradiction. 

\textit{Step 2.}
One has $\norm{U_n}_\infty < +\infty$ for any fixed $n$. 

\noindent Indeed, if there exists $\bar n$ such that $\norm{U_{\bar n}}_\infty=+\infty$, then there exists $x_n \in \Omega$ and $u_n \in A_{\bar n}$ such that $u_n(x_n) \to +\infty$. Then $u_n \in A_n$ for any $n \ge \bar n$,  and $\norm{u_n}_\infty \ge u_n(x_n) \to +\infty$, a contradiction. 

\textit{Step 3.}
We show that there exists a constant $C>0 $ such that $\norm{U_n}_\infty \le C$ uniformly in $n$. 

\noindent Notice that $\norm{U_n}_\infty \le \norm{U_{n+1}}_\infty$ and hence if it is not bounded, then  $\norm{U_n}_\infty \to \infty$, thus $\norm{u_n}_\infty \to \infty$ for a sequence $u_n \in A_n$, a contradiction. 

\textit{Step 4.}
One has $U_n=(U_n)^*$ is a subsolution to \eqref{equation n} such that $U_n=0$ in $\R^N \setminus \Omega$. 

\noindent Indeed, $(U_n)^*$ is a subsolution by Lemma \ref{stability2}. Moreover, since for any $u \in A_n$
\[ \begin{cases}
\I_k^+ u \ge -(1+\bar \mu_k^+ C) & \text{ in } \Omega \\
u=0 & \text{ in } \R^N \setminus \Omega,
\end{cases} \]
where $C$ is the constant found in Step 3,
by applying Lemma \ref{barrier} we have $u(x) \le \tilde C d(x)^s$, for a positive constant $\tilde C= \tilde C(\bar \mu_k^+ C, s, \Omega)$, and as a consequence $(U_n)^*=0$ in $\R^N \setminus \Omega$. Finally, by maximality of $U_n$, we conclude $U_n=(U_n)^*$. 

\textit{Step 5.} Conclusion of the proof of the claim.

\noindent 
By using the same argument as in the proof of Lemma \ref{perron} (in particular the bump construction), we prove that $(U_n)_*$ is a supersolution to \eqref{equation n}, 
which implies that $(U_n)_*+\varepsilon$ is a supersolution of
\[ \I_k^+ w +\left(\bar \mu_k^+ + \frac1n \right) w=0 \text{ in } \Omega \]
if $n$ is sufficiently big, and $\varepsilon$ is sufficiently small. Also, $(U_n)_*+\varepsilon>0$ in $\overline \Omega$, which contradicts the definition of $\bar \mu_k^+$. 
\end{proof}

\begin{lemma}\label{mu=bar mu}
Let $\Omega$ be a convex domain. Then $\mu_k^+=\bar \mu_k^+$.
\end{lemma}
\begin{proof}
Fix any $\varepsilon>0$. Let $v \in LSC(\Omega)\cap L^\infty(\R^N)$ such that  $v>0$  in $ \Omega$, $v \ge 0$ in $\R^N$, and $\I_k^+ v + (\mu_k^+-\varepsilon) v \le 0$ in $ \Omega$. 
Fix $x_0 \in \Omega$, and observe that 
\[ \tilde v(x)= v\left( \frac{x+\varepsilon x_0}{1+\varepsilon} \right) \]
satisfies
\[ \I_k^+ \tilde v + \frac{\mu_k^+-\varepsilon}{(1+\varepsilon)^{2s}} \tilde v \le 0 \text{ in } \Omega. \]
Also, $\tilde v >0$ in $\overline \Omega$, as $\Omega$ is convex. Thus, 
\[ \bar \mu_k^+ \ge \frac{\mu_k^+-\varepsilon}{(1+\varepsilon)^{2s}} \]
from which letting $\varepsilon \to 0$ we have $\mu_k^+\le \bar \mu_k^+$, and by definition equality holds. 
\end{proof}

\begin{theorem}\label{max principle}
Let $\Omega$ be a uniformly convex  domain. The operator
\[ I^+_k + \mu \]
satisfies the maximum principle if and only if $\mu < \mu_k^+ < +\infty$, and correspondingly 
\[ I_k^- + \mu \]
satisfies the maximum principle for any $\mu\in\mathbb R$. 
\end{theorem}
\begin{proof}
Immediately follows from Theorems \ref{max principle bar}-\ref{subsol mu}, Proposition \ref{mu-inf} and Lemma \ref{mu=bar mu}. 
\end{proof}

\section{H\"older estimates}\label{sec:holder}

\begin{proposition}\label{holder}
Let $u$ satisfy 
\begin{equation}\label{2406eq11}  \begin{cases}
\I_1^+ u(x) =f(x) & \text{in $\Omega$} \\
u=0 & \text{in $\R^N \setminus \Omega$}, 
\end{cases} \end{equation}
where $\Omega$ is a uniformly convex domain. If $s>\frac12$, then $u$ is H\"older continuous of order $2s-1$ in $\R^N$. 
\end{proposition}
\begin{proof}
It is sufficient to show that for any $x, y \in \overline{\Omega}$ such that $\abs{x-y} < \rho$, where $\rho=\rho(s,\left\|f\right\|_\infty)$ is a positive constant to be determined, then 
\begin{equation}\label{lip} u(x)-u(y) \le L \abs{x-y}^{2s-1} \end{equation}
with $L=L(\Omega, \norm{u}_\infty, \norm{f}_\infty,s)$. 
Fix $\theta\in(s,2s)$ and consider
\[ w(|x|)=-\abs{x}^{2s-1} + \abs{x}^\theta, \]
which has a minimum in 
\[ r_0= \left( \frac{2s-1}{\theta} \right)^{\frac{1}{\theta-2s+1}}\,. \]
Set
\begin{equation}\label{eq:v} v(x)= \begin{cases}
w(|x|) &\text{if } \abs{x} \le r_0\\
w(r_0) &\text{if } \abs{x} > r_0.
\end{cases}\end{equation}
We claim that there exists $\bar\rho=\bar\rho(s,\left\|f\right\|_\infty)$ sufficiently small such that 
\begin{equation}\label{claim holder} \I_1^+ v(x) \ge \left\|f\right\|_\infty\quad\forall x\in B_{\bar\rho}(0)\backslash\left\{0\right\}. \end{equation}  
In order to show \eqref{claim holder}, we fix $x \in B_{\bar\rho}(0)$, where $\bar\rho < r_0$ will be chosen later, and notice that it is sufficient to make computations in the parallel direction $I_{\hat x}v$, thus 
\begin{align*}
I_{\hat x} v (x)&= C_s \int_{0}^{+\infty} \frac{\delta(v, x, \tau \hat x)}{\tau^{1+2s}} \, d\tau  \\
&= C_s \Big( \int_{0}^{r_0 - \abs{x}} \frac{\delta(w, x, \tau \hat x)}{\tau^{1+2s}}\, d\tau + \int_{r_0 - \abs{x}}^{r_0 +\abs{x}} \frac{w(|x-\tau \hat x|) +w(r_0)-2 w(|x|) }{\tau^{1+2s}}\, d\tau \\
&\hspace{8cm}+2 \int_{r_0 + \abs{x}}^{+\infty} \frac{w(r_0) - w(x) }{\tau^{1+2s}}\, d\tau  \Big).
\end{align*}
We now add and subtract the integral
\[ C_s  \int_{r_0 - \abs{x}}^{+\infty} \frac{\delta(w, x, \tau \hat x)}{\tau^{1+2s}}\, d\tau ,  \]
and as a result 
\[ I_{\hat x}v(x)=C_s (J_1+J_2+J_3), \]
where 
\[ J_1= \int_{0}^{+\infty} \frac{\delta(w, x, \tau \hat x)}{\tau^{1+2s}}\, d\tau =-  \int_{0}^{+\infty} \frac{\delta(\abs{x}^{2s-1}, x, \tau \hat x)}{\tau^{1+2s}}\, d\tau  + \int_{0}^{+\infty} \frac{\delta(\abs{x}^{\theta}, x, \tau \hat x)}{\tau^{1+2s}} \, d\tau, \]
\[ J_2= \int_{r_0 + \abs{x}}^{+\infty} \frac{w(r_0) - w(|x-\tau \hat x|) }{\tau^{1+2s}} \, d\tau \]
and 
\[J_3= \int_{r_0 -\abs{x}}^{+\infty} \frac{w(r_0) - w(|x+\tau \hat x|) }{\tau^{1+2s}}\, d\tau. \]

Recall that 
\[ J_1= c_\theta \abs{x}^{\theta-2s}, \]
where $c_\theta>0$ as $\theta > 2s-1$, see Lemma 3.6 in \cite{BGT}. 
Moreover, using $w(r_0)<0$, 
\begin{align*} J_2&=
 \int_{r_0 + \abs{x}}^{+\infty} \frac{w(r_0)}{\tau^{1+2s}}\,d\tau - \int_{r_0 + \abs{x}}^{+\infty} \frac{w(|x-\tau \hat x|) }{\tau^{1+2s}}\,d\tau\\
&= \frac{1}{2s} w(r_0) (r_0 +\abs{x})^{-2s} +\int_{r_0 + \abs{x}}^{+\infty} \frac{\abs{\abs{x} - \tau}^{2s-1} - \abs{\abs{x} - \tau}^\theta}{\tau^{1+2s}}\,d\tau \\
 &\ge \frac{1}{2s}  w(r_0) (r_0 +\abs{x})^{-2s} - \abs{x}^{\theta-2s} \int_{r_0/\abs{x}+1}^{+\infty} \frac{\abs{1-\tau}^\theta}{\tau^{1+2s}}\,d\tau\\
 & \ge  \frac{1}{2s}  w(r_0) r_0^{-2s} - \abs{x}^{\theta-2s} \int_{r_0/\bar\rho+1}^{+\infty} \frac{\abs{1-\tau}^\theta}{\tau^{1+2s}}\,d\tau\\
&\ge  \frac{1}{2s}  w(r_0) r_0^{-2s} - \abs{x}^{\theta-2s} \int_{r_0/\bar\rho+1}^{+\infty}\tau^{\theta-1-2s}\,d\tau\\
&= \frac{1}{2s}  w(r_0) r_0^{-2s}-\frac{\abs{x}^{\theta-2s}}{2s-\theta}{\left(1+\frac{r_0}{\bar\rho}\right)}^{\theta-2s}\,.
 \end{align*}
Similarly for $\bar\rho<\frac{r_0}{2}$
\begin{align*} J_3&= \int_{r_0- \abs{x}}^{+\infty} \frac{w(r_0)}{\tau^{1+2s}}\,d\tau  - \int_{r_0- \abs{x}}^{+\infty}\frac{w(|x+\tau \hat x|) }{\tau^{1+2s}}\,d\tau  \\
& \ge \frac{1}{2s}  w(r_0) (r_0-\abs{x})^{-2s} - \abs{x}^{\theta-2s} \int_{r_0/\abs{x}- 1}^{+\infty}\frac{\abs{1+\tau}^\theta }{\tau^{1+2s}}\,d\tau  \\
& \ge \frac{1}{2s}  w(r_0) (r_0-\bar\rho)^{-2s} - \abs{x}^{\theta-2s} \int_{r_0/\bar\rho- 1}^{+\infty}\frac{\abs{1+\tau}^\theta }{\tau^{1+2s}}\,d\tau \\
&\ge \frac{1}{2s}  w(r_0) (r_0-\bar\rho)^{-2s} - 2^\theta\abs{x}^{\theta-2s} \int_{r_0/\bar\rho- 1}^{+\infty}\tau^{\theta-1-2s}\,d\tau\\
&=\frac{1}{2s}  w(r_0) (r_0-\bar\rho)^{-2s}-\frac{2^\theta\abs{x}^{\theta-2s} }{2s-\theta}{\left(\frac{r_0}{\bar\rho}-1\right)}^{\theta-2s}\,.
\end{align*}

Summing up, 
\begin{multline*} I_{\hat x}v(x) \ge C_s   \abs{x}^{\theta-2s} \Big( c_\theta -  \frac{1}{2s-\theta}{\left(1+\frac{r_0}{\bar\rho}\right)}^{\theta-2s} -\frac{2^\theta}{2s-\theta}{\left(\frac{r_0}{\bar\rho}-1\right)}^{\theta-2s} \\+\frac{1}{2s}  {\bar\rho}^{2s-\theta} w(r_0) \left( r_0^{-2s}+ (r_0-\bar\rho)^{-2s} \right) \Big). \end{multline*}

Since the expression in parenthesis tends to $c_\theta>0$ as $\bar\rho\to0$, then we can pick $\bar\rho=\bar\rho(s,\left\|f\right\|_\infty)$ sufficiently small such that 
\begin{equation}\label{2406eq6} \I_1^+ v(x) \ge \norm{f}_\infty \text{ in } B_{\bar\rho}(0) \setminus \{0\}. \end{equation}
This shows \eqref{claim holder}. 

Let $x_0, y_0 \in \overline \Omega$ with $\abs{x_0-y_0}< \bar\rho$ and take
\[ v_{y_0}(x)=u(y_0) + L v(x-y_0) \quad x \in B_{\bar\rho}(y_0), \]
where $L>0$. We want to prove that  there is $L=L(\Omega,\left\|u\right\|_\infty,\left\|f\right\|_\infty,s)$ sufficiently large such that 
\begin{equation}\label{2406eq5}
v_{y_0}(x_0)\leq u(x_0).
\end{equation}
This readily implies \eqref{lip} since $v_{y_0}(x_0)\geq u(y_0)-L|x_0-y_0|^{2s-1}$ and  $x_0,y_0$ are arbitrary points of $\overline\Omega$ with $\abs{x_0-y_0}< \bar\rho$.

 To obtain \eqref{2406eq5} we make use of the comparison principle, see Theorem \ref{comparison}, in $\Omega\cap B_{\bar\rho}(y_0)\backslash\left\{y_0\right\}$. By \eqref{2406eq6}, if $L\geq1$ then
\[ \I_1^+v_{y_0}(x) \ge \norm{f}_\infty \text{ in } B_{\bar\rho}(y_0) \setminus \{y_0\}, \]
hence $v_{y_0}$ is a subsolution of $\I_1^+v=f(x)$ in $B_{\bar\rho}(y_0) \setminus \{y_0\}$. As far as the exterior boundary condition is concerned, first notice that by definition $v_{y_0}(y_0)=u(y_0)$. Now let $x \in\R^N\backslash B_{\bar\rho}(y_0)$. Since the function $v(x)$ is radially decreasing it turns out that 
\[ v(x-y_0) \le - \bar\rho^{2s-1} + \bar\rho^\theta \]
and, for
\begin{equation}\label{2406eq7}
L\geq\frac{2\left\|u\right\|_\infty}{{\bar\rho}^{2s-1}-{\bar\rho}^\theta},
\end{equation}
that
\[ v_{y_0}(x)=u(y_0)+Lv(x-y_0) \le u(y_0)-L\bar\rho^{2s-1} + L\bar\rho^\theta \le u(y_0)-2 \norm{u}_\infty \le u(x). \]
It remains to prove the inequality $v_{y_0}(x)\leq u(x)$ for  $x \in  \overline{B_{\bar\rho}(y_0)} \cap \Omega^c$. For this we recall that by Lemma \ref{barrier} there exists a positive constant $C=C(\Omega,\left\|f\right\|_\infty,s)$ such that 
\begin{equation}\label{2406eq8}
u(y_0) \le C d(y_0)^s\leq C|x-y_0|^s\,.
\end{equation}
Notice that the function $r\in(0,+\infty)\mapsto r^{s-1} - r^{\theta - s}$ is decreasing, thus
\begin{equation}\label{2406eq10}
 r^{s-1} - r^{\theta - s}\geq \bar\rho^{s-1} - \bar\rho^{\theta - s}\quad\forall r\in(0,\bar\rho].
\end{equation}
Using \eqref{2406eq10} with $r=|x-y_0|$ and \eqref{2406eq8} we obtain, for  $x \in  \overline{B_{\bar\rho}(y_0)} \cap \Omega^c$, that 
\[ \begin{split} u(x)=0 &\ge u(y_0) - C\abs{x-y_0}^s \\
&\ge u(y_0) - L\abs{x-y_0}^{2s-1} + L \abs{x-y_0}^\theta=v_{y_0}(x)\end{split} \]
provided
\begin{equation}\label{2406eq9}
L\geq\frac{C}{{\bar\rho}^{s-1}-{\bar\rho}^{\theta-s}}\,.
\end{equation}
Summing up, by \eqref{2406eq7} and\eqref{2406eq9}, if 
$$
L\geq\max\left\{\frac{2\left\|u\right\|_\infty}{\bar\rho^{2s-1}-\bar\rho^\theta},\frac{C}{\bar\rho^{s-1}-\bar\rho^{\theta-s}},1\right\},
$$
then by comparison we conclude that \eqref{2406eq5} holds, as we wanted to show.
 \end{proof}

Let us point out that, as in the local setting (see \cite[Section 3]{BGI}), 
 the uniform convexity of $\Omega$ was just exploited in the proof of Proposition \ref{holder} to get \eqref{2406eq8}, hence to apply comparison principle up to the boundary. Moreover, in order to obtain interior H\"older estimates is in fact sufficient to assume the function $u$ to be only supersolution.

\begin{proposition} 
Let $\Omega$ be a bounded domain of $\R^N$, and let $s > \frac12$. Then:
\begin{itemize}
	\item[i)] for any compact $K\subset\Omega$ and any supersolution $u$ of \eqref{2406eq11}, there exists a positive constant $C=C(K,\Omega,\left\|u\right\|_\infty,\left\|f\right\|_\infty, s)$ such that $\left\|u\right\|_{C^{0,2s-1}(K)}\leq C$;
	\item[ii)] any supersolution $u$ which satisfies \eqref{2406eq12} is $(2s-1)$-H\"older continuous in $\overline\Omega$.
\end{itemize}
\end{proposition}

In the next theorem we obtain global H\"older equicontinuity of sequences of solutions with uniformly bounded right hand sides. We shall use it in the next section for the existence of a principal eigenfuntion.

\begin{theorem}\label{equiHolder}
Let $s > \frac12$, and let $u_n\in C(\overline\Omega)\cap L^\infty(\R^N)$ be solutions of 
\begin{equation*}
\begin{cases}
\I^+_1u_n=f_n(x) & \text{in $\Omega$}\\
u_n=0  & \text{in $\R^N\backslash\Omega$},
\end{cases}
\end{equation*}
where the domain $\Omega$ is uniformly convex and $f_n\in C(\Omega)$ for any $n\in\mathbb N$. Assume that there exists a positive constant $D$ such that
\begin{equation}\label{2506eq1}
\sup_{n\in\mathbb N}\left\|u_n\right\|_{L^\infty{(\R^N\backslash\Omega)}}+\left\|f_n\right\|_{L^\infty{(\Omega)}}\leq D.
\end{equation}
Then there exists $ \tilde D=\tilde D(D,\Omega,s)>0$  such that 
\begin{equation}\label{2506eq4}
\sup_{n\in\mathbb N}\left\|u_n\right\|_{C^{0,2s-1}(\mathbb R^N)}\leq \tilde D.
\end{equation}
\end{theorem}
\begin{proof}
We start by showing that $\sup_n\left\|u_n\right\|_{L^\infty(\R^N)}<+\infty$. Let $R$, just depending on $\Omega$, be such that $B_R(0)\supseteq\Omega$ and consider the function
$$
\varphi(x)=\frac{D}{C_s\beta(1-s,s)}{\left(R^2-|x|^2\right)}^s_+.
$$
By Lemma \ref{funzione barriera}, $\varphi$ solves
\begin{equation*}
\begin{cases}
\I^+_1\varphi=-D & \text{in $\Omega$}\\
\varphi\geq0 & \text{in $\R^N\backslash\Omega$}.
\end{cases}
\end{equation*}
For any $n\in\mathbb N$, using \eqref{2506eq1}, $u_n$ is solution of 
\begin{equation*}
\begin{cases}
\I^+_1u_n\geq-D & \text{in $\Omega$}\\
u_n=0 & \text{in $\R^N\backslash\Omega$}.
\end{cases}
\end{equation*}
Hence by the comparison Theorem \ref{comparison} we get 
\begin{equation}\label{2506eq2}
u_n(x)\leq \varphi(x)\leq \frac{DR^{2s}}{C_s\beta(1-s,s)} \quad\forall x\in\Omega.
\end{equation}
In a similar fashion we also obtain
\begin{equation}\label{2506eq3}
u_n(x)\geq-\frac{DR^{2s}}{C_s\beta(1-s,s)} \quad\forall x\in\Omega.
\end{equation}
From  \eqref{2506eq2}-\eqref{2506eq3} and \eqref{2506eq1}  we infer that $\sup_n\left\|u_n\right\|_{L^\infty(\R^N)}<+\infty$, in fact
$$
\sup_n\left\|u_n\right\|_{L^\infty(\R^N)}\leq \max\left\{D,\frac{DR^{2s}}{C_s\beta(1-s,s)}\right\}.
$$
Arguing as in the proof of Proposition \ref{holder}, with the same notations there used, and $v$ as defined in \eqref{eq:v}, we can pick $\bar\rho=\bar\rho(s,D)$ such that 
$$
\I^+_1v(x)\geq D\quad\text{in $B_{\bar\rho}(0)\backslash\left\{0\right\}$}.
$$
Moreover by Lemma \ref{barrier} there exists a positive constant $C=C(\Omega,D,s)$ such that 
\begin{equation*}
u_n(x) \le C d(x)^s\quad\forall x\in\overline\Omega.
\end{equation*}
Hence by taking
$$
L\geq\max\left\{\frac{2\sup_n\left\|u_n\right\|_\infty}{\bar\rho^{2s-1}-\bar\rho^\theta},\frac{C}{\bar\rho^{s-1}-\bar\rho^{\theta-s}},1\right\}
$$
we conclude that for any $n\in\mathbb N$ and  any $x,y\in\overline\Omega$ such that $|x-y|\leq\bar\rho$ then $$u_n(x)-u_n(y)\leq L|x-y|^{2s-1}.$$ 
This readily implies \eqref{2506eq4}.
\end{proof}

\section{Existence of  a principal eigenfunction}\label{sec:existence}

%%%
The main result of this section is the following 
\begin{theorem}\label{existence}
Let $\Omega$ be a uniformly convex domain, and let $s > \frac12$. Then there exists a positive function $\psi_1 \in C^{0,2s-1}(\overline \Omega)$ such that 
\begin{equation}\label{DirEigen} \begin{cases}
\I_1^+ \psi_1 + \mu_1^+ \psi_1=0 &\text{ in } \Omega \\
\psi_1=0 &\text{ in } \R^N \setminus \Omega.
\end{cases} \end{equation}
\end{theorem}

For this we first prove the solvability of the Dirichlet problem \lq\lq below the principal eigenvalue\rq\rq.

\begin{theorem}\label{existencebelow}
Let $\Omega$ be a uniformly convex domain, $s > \frac12$, and let $f\in C(\Omega)\cap L^\infty(\Omega)$. Then there exists a solution $u\in C^{0,2s-1}(\overline\Omega)$ of 
\begin{equation}\label{2506eq6}
\begin{cases}
\I^+_1u+\mu u=f(x) & \text{in $\Omega$}\\
u=0  & \text{in $\R^N\backslash\Omega$},
\end{cases}
\end{equation}
in the following cases:
\begin{itemize}
	\item[(i)] for any $\mu$ if $f\geq0$
	\item[(ii)] for any $\mu<\mu^+_1$.
\end{itemize}
In the case $\mu<\mu^+_1$ the solution is unique.
\end{theorem}
\begin{proof}
 We can assume $\mu>0$, since the arguments of the proof of Theorem \ref{lem dirichlet} continue to apply for $\I^\pm_k+\mu u$ when $\mu\leq0$. \\

(i) Let $w_1=0$ and define iteratively $w_{n+1} \in C(\R^N)$ as the solution, obtained by Theorem \ref{lem dirichlet}, of
\begin{equation}\label{limit 1} \begin{cases} 
\I_1^+ w_{n+1} = f(x) - \mu w_n(x) &\text{ in } \Omega \\
w_{n+1}=0 &\text{ in } \R^N \setminus \Omega.
\end{cases} \end{equation}
Note that the sequence $(w_n)_n$ is nonincreasing and in particular $w_n \le 0$ for any $n$. 
Indeed, since $f\geq0$ then  $w_2\leq0=w_1$ by  Theorem \ref{comparison}. Moreover assuming by induction $w_{n+1}\le w_n$, one has 
\[ \I_1^+ w_{n+2} = f- \mu w_{n+1} \ge f- \mu w_{n} = \I_1^+ w_{n+1}, \]
hence again by comparison $w_{n+2} \le w_{n+1}$. \\
We now show that $\sup_n \norm{w_n}_\infty < +\infty$. If this is true, then in view of Theorem \ref{equiHolder}, the sequence $(w_n)_n$ converges uniformly in $\R^N$ to $u\in C^{0,2s-1}(\R^N)$, and passing to the limit in \eqref{limit 1} we conclude, exploiting Lemma \ref{stability1}. 
Let us assume by contradiction that $\lim_{n\to+\infty} \norm{w_n}_\infty = +\infty$, and let $v_n=\frac{w_n}{\norm{w_n}}$. Then 
\[ \begin{cases}
\I_1^+ v_{n+1}=\frac{f(x)}{\norm{w_{n+1}}} - \mu \frac{\norm{w_n}}{\norm{w_{n+1}}}v_n(x)  &\text{ in } \Omega \\
v_{n+1}=0 &\text{ in } \R^N \setminus \Omega.
\end{cases} \] 
Then again by the H\"older estimate \eqref{2506eq4} the sequence $(v_n)_n$ converges uniformly, up to a subsequence,   to a function $v \le 0$. Since, up to extract a further subsequence, $\frac{\norm{w_n}}{\norm{w_{n+1}}} \to \tau \le 1$, we may pass to the limit to get
\[ \begin{cases}
\I_1^+ v +\mu \tau v =0 &\text{ in } \Omega \\
v =0 &\text{ in } \R^N \setminus \Omega.
\end{cases} \] 
Now since $\I^-_1(-v)+\mu\tau(-v)=0$ in  $\Omega$, by Theorem \ref{max principle} we infer that $v$  in fact vanishes everywhere. This is in contradiction to $\left\|v\right\|_\infty=1$.\\

(ii) We first claim that there exists a nonnegative solution $\overline w\in C^{0,2s-1}(\R^N)$ of 
\begin{equation}\label{2506eq5}
 \begin{cases}
\I_1^+\overline w + \mu  \overline w=-\norm{f}_\infty &\text{ in } \Omega \\
\overline w=0 &\text{ in } \R^N \setminus \Omega.
\end{cases} \end{equation}
As above, we define $w_1=0$ and $w_{n+1}$ be the solution of
\begin{equation*} \begin{cases} 
\I_1^+ w_{n+1} = -\norm{f}_\infty - \mu w_n(x) &\text{ in } \Omega \\
w_{n+1}=0 &\text{ in } \R^N \setminus \Omega.
\end{cases} \end{equation*}
The sequence $(w_n)_n$ is nondecreasing. Using now that $\mu<\mu^+_1$ we also infer that $\sup_n \norm{w_n}_\infty < +\infty$. Then, by Theorem \ref{equiHolder}, $w_n$ converges uniformly in $\R^N$
to a function $\overline w\in C^{0,2s-1}(\R^N)$ which  is solution of \eqref{2506eq5}.

For the general case, let us denote by $\underline w$ the solution of 
\begin{equation*}
 \begin{cases}
\I_1^+\underline w + \mu  \underline w=\norm{f}_\infty &\text{ in } \Omega \\
\underline w=0 &\text{ in } \R^N \setminus \Omega.
\end{cases}
\end{equation*}
obtained in i). Notice that $\underline w\leq0\leq\overline w$.

Now let us define $u_1= \underline{w}$ and let $u_{n+1}$ be the solution of
\[ \begin{cases}
\I_1^+ u_{n+1} = f(x) - \mu u_n &\text{ in } \Omega \\
u_{n+1} =0 &\text{ in } \R^N \setminus \Omega.
\end{cases} \]
We want to show that $\underline{w} \le u_n \le \overline w$. This is true for $n=1$. Let us assume by induction that this holds true at level $n$, and notice that
\[ 
\I_1^+ u_{n+1}\geq- \norm{f}_\infty - \mu \overline{w}=\I_1^+ \overline w\quad\text{in $\Omega$}
 \]
and similarly
\[ \I_1^+ u_{n+1}\leq \norm{f}_\infty - \mu \underline{w}=\I_1^+ \underline w\quad\text{in $\Omega$}.\]
Hence by comparison we have $\underline{w} \le u_{n+1} \le \bar w$.
As a consequence, the sequence $(u_n)_n$ is bounded in $C^{0,2s-1}(\R^N)$ and up to a subsequence it converges uniformly to a function $u\in C^{0,2s-1}(\R^N)$ which is the desired solution.

It remains to show that \eqref{2506eq6} has at most one solution. For this notice that if $u$ and $v$ are respectively sub and supersolution of 
$
\I^+_1u+\mu u=f
$ in $\Omega$, then the difference $w=u-v$ is a viscosity subsolution of 
$$\I^+_1w+\mu w=0\quad\text{in $\Omega$.}$$
This easily follows if at least one between $u$ and $v$ are in $C^2(\Omega)$. Instead, if $u$ and $v$ are merely semicontinuous, then using the doubling variables technique, as  in the proof of Theorem \ref{comparison} with minor changes, we obtain the result.
Hence, if $u_1$ and $u_2$ are solutions of \eqref{2506eq6} then $w=u_1-u_2$ solves 
\begin{equation*}
\begin{cases}
\I^+_1w+\mu w\geq0 & \text{in $\Omega$}\\
w=0  & \text{in $\R^N\backslash\Omega$}.
\end{cases}
\end{equation*}
By Theorem \ref{max principle} we infer that $u_1\leq u_2$. Reversing the role of $u_1$ and $u_2$ we conclude that $u_1=u_2$.
\end{proof}

We are now in position to give the proof of Theorem \ref{existence}.

\begin{proof}[Proof of Theorem \ref{existence}]
In view of Theorem \ref{existencebelow}, for any $n\in\mathbb N$ there exists a solution $w_n\in C^{0,2s-1}(\overline\Omega)$ of  
\[ \begin{cases}
\I_1^+ w_n +(\mu_1^+ -\frac1n ) w_n=-1 &\text{ in } \Omega \\
w_n>0 &\text{ in } \Omega \\
w_n =0 &\text{ in } \R^N \setminus \Omega.
\end{cases} \]
We claim that $\sup_n \norm{w_n}=+\infty$. If not, we can pick $j\in\mathbb N$ such that $j\geq2\sup_n \norm{w_n}$. Hence $w_j$ solves
\[ \begin{cases}
\I_1^+ w_j +(\mu_1^+ +\frac1j ) w_j\leq0 &\text{ in } \Omega \\
w_j>0 &\text{ in } \Omega \\
w_j =0 &\text{ in } \R^N \setminus \Omega.
\end{cases} \]
This contradicts the maximality of $\mu_1^+$, and proves that $\sup_n \norm{w_n} =+\infty$. Up to a subsequence we may assume $\lim_n \norm{w_n}= +\infty$, and we can introduce the functions $z_n=\frac{w_n}{\norm{w_n}}$, which turn out to be solutions of
\[ \begin{cases}
\I_1^+ z_n + \left (\mu_1^+ -\frac1n\right ) z_n = - \frac{1}{\norm{w_n}}&\text{ in } \Omega \\
z_n =0 &\text{ in } \R^N \setminus \Omega.
\end{cases} \]
Using the estimate \eqref{2506eq4}, the sequence $(z_n)_n$ converges uniformly  to a function $\psi_1\in C^{0,2s-1}(\overline\Omega)$ which is solution of \eqref{DirEigen}. Moreover $\psi_1 \geq0$ in $\Omega$ by construction and $\left\|\psi_1 \right\|_\infty=1$. By the strong minimum principle, see Theorem \ref{SMP}-iii), we conclude that $\psi_1>0$ in $\Omega$.
\end{proof}

\end{document}